\documentclass[12pt,twoside,leqno]{article}
\usepackage[T1]{fontenc}
\usepackage{amsfonts}
\usepackage{amsmath,amsthm}
\usepackage{amssymb,latexsym}
\usepackage{enumerate}
\theoremstyle{plain}
\newtheorem{theorem}{Theorem}[section]
\newtheorem{lemma}[theorem]{Lemma}
\newtheorem{proposition}[theorem]{Proposition}
\newtheorem{corollary}[theorem]{Corollary}

\theoremstyle{definition}

\newtheorem{definition}[theorem]{Definition}
\theoremstyle{remark}
\newtheorem{remark}[theorem]{Remark}

\DeclareMathOperator{\fix}{{\rm fix}}

\DeclareMathOperator{\supp}{{\rm supp}}
\DeclareMathOperator{\sym}{{\rm sym}}

\DeclareMathOperator{\cl}{{\rm cl}}

\begin{document}
\title{Second-countable compact Hausdorff spaces as remainders in $\mathbf{ZF}$ and two new notions of infiniteness}
\author{Kyriakos Keremedis, Eleftherios Tachtsis and Eliza Wajch\\
Department of Mathematics, University of the Aegean\\
Karlovassi, Samos 83200, Greece\\
kker@aegean.gr\\
Department of Statistics and Actuarial-Financial Mathematics,\\
 University of the Aegean, Karlovassi 83200, Samos, Greece\\
 ltah@aegean.gr\\
Institute of Mathematics\\
Faculty of Exact and Natural Sciences \\
Siedlce University of Natural Sciences and Humanities\\
ul. 3 Maja 54, 08-110 Siedlce, Poland\\
eliza.wajch@wp.pl}
\maketitle
\begin{abstract}

In the absence of the Axiom of Choice, necessary and sufficient conditions for a locally compact Hausdorff space to have all non-empty second-countable compact Hausdorff spaces as remainders are given in $\mathbf{ZF}$. Among other independence results, the characterization of locally compact Hausdorff spaces having all non-empty metrizable compact spaces as remainders, obtained by Hatzenhuhler and Mattson in $\mathbf{ZFC}$, is proved to be independent of $\mathbf{ZF}$. Urysohn's Metrization Theorem is generalized to the following theorem: every $T_3$-space which admits a base expressible as a countable union of finite sets is metrizable. Applications to solutions of problems concerning the existence of some special metrizable compactifications in $\mathbf{ZF}$ are shown. New concepts of a strongly filterbase infinite set and a dyadically filterbase infinite set are introduced, both stemming from the investigations on compactifications. Set-theoretic and topological definitions of the new concepts are given, and their relationship with certain known notions of infinite sets is investigated in $\mathbf{ZF}$. A new permutation model is introduced in which there exists a strongly filterbase infinite set which is weakly Dedekind-finite. All $\mathbf{ZFA}$-independence results of this article are transferable to $\mathbf{ZF}$. \medskip

\noindent\textit{Mathematics Subject Classification (2010)}:03E25, 03E35, 54A35, 54D35, 54D40, 54D45, 54E35\newline 
\textit{Keywords}: Weak forms of the Axiom of Choice, compactification, remainder, Cantor set, metrizability, Urysohn's Metrization Theorem
\end{abstract}

\section{Preliminaries}
\label{s1}
\subsection{Set-theoretic framework and preliminary definitions}
\label{s1.1}
In this note, the intended context for reasoning and statements of theorems is $\mathbf{ZF}$ without any form of the Axiom of Choice $\mathbf{AC}$. Before we pass to the content of the paper, let us establish basic terminology and notation, give a list of the weaker forms of $\mathbf{AC}$ that are used in this article, and recall several known theorems we refer to in the main part of the text.

We denote by $ON$ the class of all (von Neumann) ordinal numbers. The first infinite ordinal number is denoted by $\omega$. Then $\mathbb{N}=\omega\setminus\{0\}$. If $X$ is a set, the power set of  $X$ is denoted by $\mathcal{P}(X)$. A set $X$ is called \emph{finite} if $X$ is equipotent to a member of $\omega$; otherwise $X$ is called \emph{infinite}. A set $X$ is called \emph{countable} if $X$ is equipotent to a subset of $\omega$. An infinite countable set is called \emph{denumerable}. The set of all finite subsets of $X$ is denoted by $[X]^{<\omega}$. For every set $S$ and every ordinal $\alpha$, the set $\mathcal{P}^{\alpha}(S)$ is defined by a transfinite induction on ordinals as follows: $\mathcal{P}^{0}(S)=S$, $\mathcal{P}^{\gamma}(S)=
\bigcup\limits_{\beta\in\gamma} \mathcal{P}^{\beta}(S)$ if $\gamma$ is a limit ordinal, and $\mathcal{P}^{\gamma +1}(S)
=\mathcal{P}^{\gamma}(S)\cup\mathcal{P}(\mathcal{P}^{\gamma}(S))$ (cf. \cite[p. 43]{Je}).  If $\alpha\in On$ and it is not stated otherwise, $\alpha$ denotes also the topological space $\langle \alpha, \tau\rangle$ where $\tau$ is the topology in $\alpha$ induced by the standard linear order in $\alpha$ defined as follows: for all $x,y\in\alpha$, $x\leq y$ if and only if $x\subseteq y$.

As usual, the system $\mathbf{ZF+AC}$ is denoted by $\mathbf{ZFC}$. To stress the fact that a result is proved in $\mathbf{ZF}$ (respectively, $\mathbf{ZFC}$), we shall write at the beginning of the statements of the theorems and propositions ($\mathbf{ZF}$) or ($\mathbf{ZFC}$), respectively. Apart from models of
$\mathbf{ZF}$, we refer to some permutation models of $\mathbf{ZFA}$. Basic facts about permutation models (called also Fraenkel-Mostowski models) and Pincus transfer theorems that are applied here are given, for instance, in \cite{Je}, \cite{hr}, \cite{ktw0} and \cite{pin1}-\cite{pin2}.

In the sequel,  topological or metric spaces (called \emph{spaces} in abbreviation) are denoted by boldface letters, and the underlying sets of the spaces are denoted by lightface letters. All topological notions used in this article but not defined here are standard and they can be found, for instance, in \cite{En},  \cite{w} and \cite{ch}.

For a topological space $\mathbf{X}=\langle X, \tau\rangle$ and for $Y\subseteq X$, let $\tau|_Y=\{V\cap Y: V\in\tau\}$ and let $\mathbf{Y}=\langle Y, \tau|_Y\rangle$. Then $\mathbf{Y}$ is the topological subspace of $\mathbf{X}$ such that $Y$ is the underlying set of $\mathbf{Y}$. If this is not misleading, we may denote the topological subspace $\mathbf{Y}$ of $\mathbf{X}$ by $Y$. We denote by $\cl_{\mathbf{X}}(Y)$ or by $\cl_{\tau}(Y)$  the closure of $Y$ in $\mathbf{X}$. The collection of all compact subsets of $\mathbf{X}$ is denoted by $\mathcal{K}(\mathbf{X})$.

For a metric space $\mathbf{X}=\langle X, d\rangle$, the $d$-\textit{ball with centre $x\in X$ and radius} $r\in(0, +\infty)$ is the set 
$$B_{d}(x, r)=\{ y\in X: d(x, y)<r\}.$$
 The collection 
$$\tau(d)=\{ V\subseteq X: (\forall x\in V)(\exists n\in\omega) B_{d}(x, \frac{1}{2^n})\subseteq V\}$$
is the \textit{topology in $X$ induced by $d$}. For a set $A\subseteq X$, let $\delta_d(A)=0$ if $A=\emptyset$, and let $\delta_d(A)=\sup\{d(x,y): x,y\in A\}$ if $A\neq \emptyset$. Then $\delta_d(A)$ is the \emph{diameter} of $A$ in $\langle X, d\rangle$.  If $Y\subseteq X$, then $d_Y=d\upharpoonright Y\times Y$ and $\mathbf{Y}=\langle Y, d_Y\rangle$. Then $\mathbf{Y}$ is the metric subspace of $\mathbf{X}$ such that $Y$ is the underlying set of $\mathbf{Y}$. If this is not misleading, given a metric space $\mathbf{X}=\langle X, d\rangle$, we also denote by $\mathbf{X}$ the topological space $\langle X, \tau(d)\rangle$. For every $n\in\mathbb{N}$, $\mathbb{R}^n$ denotes also $\langle \mathbb{R}^{n}, d_e\rangle$ and $\langle \mathbb{R}^n, \tau(d_e)\rangle$  where $d_e$ is the Euclidean metric on $\mathbb{R}^n$. Subsets of $\mathbb{R}^n$, if not stated otherwise, are considered as metric subspaces of $\langle \mathbb{R}^n, d_e\rangle$ or topological subspaces of $\langle\mathbb{R}^n, \tau(d_e)\rangle$. 

We recall that a (Hausdorff) \emph{compactification} of a space $\mathbf{X}=\langle X, \tau\rangle$ is an ordered pair $\langle\mathbf{Y},\gamma\rangle$ where $\mathbf{Y}$ is a  (Hausdorff) compact space and $\gamma: \mathbf{X}\to\mathbf{Y}$ is a homeomorphic embedding such that $\gamma(X)$ is dense in $\mathbf{Y}$. A compactification  $\langle \mathbf{Y}, \gamma\rangle$ of $\mathbf{X}$ and the space $\mathbf{Y}$ are usually denoted by $\gamma\mathbf{X}$. The underlying set of $\gamma\mathbf{X}$ is denoted by $\gamma X$. The subspace $\gamma X\setminus X$ of $\gamma \mathbf{X}$ is called the \emph{remainder} of $\gamma\mathbf{X}$. A space $\mathbf{K}$ is said to be a remainder of $\mathbf{X}$ if there exists a Hausdorff compactification $\gamma\mathbf{X}$ of $\mathbf{X}$ such that $\mathbf{K}$ is homeomorphic to $\gamma X \setminus X$. For compactifications  $\alpha\mathbf{X}$ and $\gamma\mathbf{X}$ of $\mathbf{X}$, we write $\gamma\mathbf{X}\leq\alpha\mathbf{X}$ if there exists a continuous mapping $f:\alpha\mathbf{X}\to\gamma\mathbf{X}$ such that $f\circ\alpha=\gamma$. If $\alpha\mathbf{X}$ and $\gamma\mathbf{X}$ are Hausdorff compactifications of $\mathbf{X}$ such that $\alpha\mathbf{X}\leq\gamma\mathbf{X}$ and $\gamma\mathbf{X}\leq\alpha\mathbf{X}$, then we write $\alpha\mathbf{X}\approx\gamma\mathbf{X}$ and say that the compactifications $\alpha\mathbf{X}$ and $\gamma\mathbf{X}$ are \emph{equivalent}. If $n\in\mathbb{N}$, then a compactification $\gamma\mathbf{X}$ of $\mathbf{X}$ is said to be an $n$-point compactification of $\mathbf{X}$ if $\gamma X\setminus X$ is an $n$-element set. If $\mathbf{X}$ is a non-compact locally compact Hausdorff space, then there exists a unique (up to $\approx$) one-point Hausdorff compactification of $\mathbf{X}$ which can be defined as follows:   

\begin{definition}
\label{s1d01}
Let $\mathbf{X}=\langle X, \tau\rangle$ be a non-compact locally compact Hausdorff space. For an element $\infty\notin X$, we define $X(\infty)=X\cup\{\infty\}$, 
$$\tau(\infty)=\tau\cup\{X(\infty)\setminus K: K\in\mathcal{K}(\mathbf{X})\}$$
\noindent and $\mathbf{X}(\infty)=\langle X(\infty), \tau(\infty)\rangle$. Then $\mathbf{X}(\infty)$ is called the \emph{Alexandroff compactification} of $\mathbf{X}$.
\end{definition}

For every non-compact locally compact Hausdorff space $\mathbf{X}$, $\mathbf{X}(\infty)$ is the unique (up to $\approx$) one-point Hausdorff compactification of $\mathbf{X}$. Clearly, $\mathbb{N}(\infty)$ is homeomorphic to $\omega+1$.  

\begin{definition}
\label{s1d0.2}
A Hausdorff compactification $\beta\mathbf{X}$ of a space $\mathbf{X}$ is called the \emph{\v Cech-Stone compactification} of $\mathbf{X}$ if, for every compact Hausdorff space $\mathbf{K}$ and every continuous mapping $f:\mathbf{X}\to\mathbf{K}$, there exists a continuous extension $\tilde{f}: \beta\mathbf{X}\to \mathbf{K}$ of $f$ over $\beta\mathbf{X}$.
\end{definition}

In $\mathbf{ZFC}$, every Tychonoff space has its \v Cech-Stone compactification; however, in a model of $\mathbf{ZF}$,  a Tychonoff space may fail to have its \v Cech-Stone compactification (cf., e.g., \cite[Theorem 3.7]{kw0}). The book \cite{ch} is a very good introduction to Hausdorff compactifications in $\mathbf{ZFC}$.  Basic facts about Hausdorff compactifications in $\mathbf{ZF}$ can be found in \cite{kw0}.

Given a collection  $\{X_j: j\in J\}$ of sets, for every $i\in J$, we denote by $\pi_i$ the projection $\pi_i:\prod\limits_{j\in J}X_j\to X_i$ defined by $\pi_i(x)=x(i)$ for each $x\in\prod\limits_{j\in J}X_j$. If $\tau_j$ is a topology in $X_j$, then $\mathbf{X}=\prod\limits_{j\in J}\mathbf{X}_j$ denotes the Tychonoff product of the topological spaces $\mathbf{X}_j=\langle X_j, \tau_j\rangle$ with $j\in J$. If $\mathbf{X}_j=\mathbf{X}$ for every $j\in J$, then $\mathbf{X}^{J}=\prod\limits_{j\in J}\mathbf{X}_j$. As in \cite{En}, for an infinite set $J$ and the unit interval $[0,1]$ of $\mathbb{R}$, the cube $[0,1]^J$ is called the \emph{Tychonoff cube}. If $J$ is denumerable, then the Tychonoff cube $[0,1]^J$ is called the \emph{Hilbert cube}. We denote by $\mathbf{2}$ the discrete space with the underlying set $2=\{0, 1\}$. If $J$ is an infinite set, the space $\mathbf{2}^J$ is called the \emph{Cantor cube}. The Cantor cube $\mathbf{2}^{\omega}$ is known to be homeomorphic to the Cantor ternary set.

We recall that if $\prod\limits_{j\in J}X_j\neq\emptyset$, then it is said that the family $\{X_j: j\in J\}$ has a choice function, and every element of $\prod\limits_{j\in J}X_j$ is called a \emph{choice function} of the family $\{X_j: j\in J\}$. A \emph{multiple choice function} of $\{X_j: j\in J\}$ is every function $f\in\prod\limits_{j\in J}\mathcal{P}(X_j)$ such that, for every $j\in J$, $f(j)$ is a non-empty finite subset of $X_j$. A set $f$ is called \emph{partial} (\emph{multiple}) \emph{choice function} of $\{X_j: j\in J\}$ if there exists an infinite subset $I$ of $J$ such that $f$ is a (multiple) choice function of $\{X_j: j\in I\}$. Given a non-indexed family $\mathcal{A}$, we treat $\mathcal{A}$ as an indexed family $\mathcal{A}=\{x: x\in\mathcal{A}\}$ to speak about a  (partial) choice function and a (partial) multiple choice function of $\mathcal{A}$.

Let  $\{X_j: j\in J\}$ be a disjoint family of sets, that is, $X_i\cap X_j=\emptyset$ for each pair $i,j$ of distinct elements of $J$. If $\tau_j$ is a topology in $X_j$ for every $j\in J$, then $\bigoplus\limits_{j\in J}\mathbf{X}_j$ denotes the direct sum of the spaces $\mathbf{X}_j=\langle X_j, \tau_j\rangle$ with $j\in J$.

\begin{definition}
\label{s1d06}
(Cf. \cite{br}, \cite{lo} and \cite{kerta}.) 
\begin{enumerate}
\item[(i)] A space $\mathbf{X}$ is said to be \emph{Loeb} (respectively, \emph{weakly Loeb}) if the family of all non-empty closed subsets of $\mathbf{X}$ has a choice function (respectively, a multiple choice function).
\item[(ii)] If $\mathbf{X}$ is a (weakly) Loeb space, then every (multiple) choice function of the family of all non-empty closed subsets of $\mathbf{X}$ is called a (\emph{weak}) \emph{Loeb function} of $\mathbf{X}$.
\end{enumerate}
\end{definition}

That spaces $\mathbf{X}$ and $\mathbf{Y}$ are homeomorphic is denoted by $\mathbf{X}\simeq\mathbf{Y}$.

\begin{definition}
\label{s1d03}
A collection $\mathcal{V}$ of subsets of a set $X$ is called:
\begin{enumerate}
\item[(i)] \emph{stable under finite unions} (respectively, \emph{ finite intersections}) if, for every pair $U,V$ of members of $\mathcal{V}$, $U\cup V\in\mathcal{V}$ (respectively, $U\cap V\in\mathcal{V}$);
\item[(ii)] a \emph{filterbase} in $X$ if $\mathcal{V}\neq\emptyset$ and, for every  pair $U, V$ of members of $\mathcal{V}$, there exists $W\in\mathcal{V}$ such that $\emptyset\neq W\subseteq U\cap V$.
\item[(iii)] a \emph{free filterbase} if $\mathcal{V}$ is a filterbase such that $\bigcap\mathcal{V}=\emptyset$.
\end{enumerate}
\end{definition}

\begin{definition}
\label{s1d04} A set $X$ is called:
\begin{enumerate} 
\item[(i)]\emph{Dedekind-finite} if there is no injection $f:\omega\to X$; \emph{Dedekind-infinite} if $X$ is not Dedekind-finite;
\item[(ii)] \emph{quasi Dedekind-finite} if $[X]^{<\omega}$ is Dedekind-finite; \emph{quasi Dedekind-infinite} if $X$ is not quasi Dedekind-finite;
\item[(iii)] \emph{weakly Dedekind-finite} if $\mathcal{P}(X)$ is Dedekind-finite; \emph{weakly Dedekind-infinite} if $\mathcal{P}(X)$ is Dedekind-infinite;
\item[(iv)] a \emph{cuf set} if $X$ is a countable union of finite sets;
\item[(v)] \emph{amorphous} if $X$ is infinite and, for every infinite subset $Y$ of $X$, the set $X\setminus Y$ is finite;
\item[(vi)] (cf. \cite{kfbi}) \emph{filterbase infinite} if there exists a family $\{\mathcal{V}_i: i\in\omega\}$ of free filterbases in $\mathbf{X}$ such that, for every pair $i,j$ of distinct elements of $\omega$, there exist $U\in\mathcal{V}_i$ and $V\in\mathcal{V}_j$ with $U\cap V=\emptyset$; \emph{filterbase finite} if $X$ is not filterbase infinite.
\end{enumerate}
\end{definition}

In the following definition, we introduce two new concepts that are stronger than the known concept of a filterbase infinite set.

\begin{definition}
\label{s1d:fbi}
A set $A$ is called:
\begin{enumerate}
\item[(a)] \emph{strongly filterbase infinite} if there exists a family $\mathcal{V}=\{\mathcal{V}_i: i\in\omega\}$ such that, for every  $i\in\omega$, $\mathcal{V}_i$ is a filterbase in $A$ such that:
\begin{enumerate}
\item[(i)] for every $i\in\omega$,  $\mathcal{V}_i$ is stable under finite unions and finite intersections, and  each member of $\mathcal{V}_i$ is infinite;
\item[(ii)] for every $i\in\omega$ and for each pair $U,V$ of members of $\mathcal{V}_i$, the set $U\setminus V$ is finite;
\item[(iii)]  for every $i\in\omega$ and every $j\in\omega\setminus\{i\}$, there exist $V\in \mathcal{V}_i$ and $W\in\mathcal{V}_j$ such that $U\cap V$ is finite;
\item[(iv)] for every $n\in\omega$ and $V_i\in\mathcal{V}_i$ with $i\in n$, the set $A\setminus\bigcup\limits_{i\in n}V_i$ is infinite;
\end{enumerate}
\item[(b)] \emph{dyadically filterbase infinite} if there exists a family $\mathcal{V}=\{\mathcal{V}^n_i: n\in\mathbb{N}, i\in \{1,\ldots, 2^n\}\}$ such that, for every $n\in\mathbb{N}$, the following conditions are satisfied:
\begin{enumerate}
\item[(v)] for every $i\in\{1,\ldots, 2^n\}$, $\mathcal{V}_i^n$ is a filterbase in $\mathbf{X}$ such that $\mathcal{V}_i^n$ is stable under finite unions and finite intersections, and each member of $\mathcal{V}_i^n$ is infinite;
\item[(vi)]  for every $i\in\{1,\ldots, 2^n\}$ and for any $U, V\in\mathcal{V}_i^n$, $U\setminus V$ is finite;
\item[(vii)] for every pair $i,j$ of distinct elements of $\{1,\ldots, 2^n\}$, for any $W\in \mathcal{V}_{i}^n$ and $G\in\mathcal{V}_j^n$,  there exist $U\in \mathcal{V}_{2i-1}^{n+1}, V\in \mathcal{V}_{2i}^{n+1}$, such that the sets $(U\cup V)\setminus W$ and $( U\cup V)\cap G)$ are both finite; 
\item[(viii)] if, for every $i\in\{1,\ldots, 2^n\}$,   $V_i\in\mathcal{V}_i^n$, then $A\setminus\bigcup\limits_{i=1}^{2^n}V_i$ is finite;
\end{enumerate}
\item[(c)] \emph{strongly} (respectively, \emph{dyadically}) \emph{filterbase finite} if $A$ is not strongly (respectively, dydadically) filterbase infinite.  
\end{enumerate}
\end{definition}

\begin{remark}
\label{s1r:def}
That Definition \ref{s1d:fbi} stems from our investigations on compactifications is shown in Section \ref{s4} where the following topological characterizations of strongly (respectively, dyadically) filterbase infinite sets are obtained: a set $X$ is strongly (respectively, dyadically) filterbase infinite if and only if $\omega+1$ (respectively, $\mathbf{2}^{\omega}$) is a remainder of the discrete space $\langle X, \mathcal{P}(X)\rangle$ (cf. Corollaries \ref{s2c06}(i) and \ref{s4:def}).
\end{remark}

\begin{definition}
\label{s1d05}
\begin{enumerate}
\item[(i)] A space $\mathbf{X}$ is called a \emph{cuf} (respectively, an \emph{amorphous}) space if its underlying set $X$ is a cuf (respectively, an amorphous) set.
\item[(ii)] A base $\mathcal{B}$ of a space $\mathbf{X}$ is called a \emph{cuf base} if $\mathcal{B}$ is a cuf set.
\end{enumerate}
\end{definition}

\subsection{The list of forms weaker than $\mathbf{AC}$}
\label{s1.2}

In this subsection, for readers' convenience, we define and denote most of the weaker forms of $\mathbf{AC}$ used directly in this paper. If a form is not defined in the forthcoming sections, its definition can be found in this subsection. For the known forms given in \cite{hr}, we quote in their statements the form number under which they are recorded in \cite{hr}.

\begin{definition}
\label{s1d7}
\begin{enumerate}
\item $\mathbf{IQDI}$: Every infinite set is quasi Dedekind-infinite.
\item $\mathbf{IWDI}$ (\cite[Form 82]{hr}): Every infinite set is weakly Dedekind-infinite. 
%\item $\mathbf{IDI}$ ( \cite[Form 9]{hr}): Every infinite set is Dedekind-infinite.
\item $\mathbf{CAC}$ (\cite[Form 8]{hr}): Every denumerable family of non-empty sets has a choice function.
\item $\mathbf{CAC}_{fin}$ ( \cite[Form 10]{hr}): Every denumerable family of non-empty finite sets has a choice function.
\item $\mathbf{AC}_{fin}$ (\cite[Form 62]{hr}): Every non-empty family of non-empty finite sets has a choice function.
%\item $\mathbf{MC}$ (the Axiom of Multiple Choice, \cite[Form 67]{hr}): Every non-empty family of non-empty sets has a multiple choice function.
\item $\mathbf{CMC}$ (the Countable Axiom of Multiple Choice, \cite[Form 126]{hr}): Every denumerable family of non-empty sets has a multiple choice function.
\item $\mathbf{NAS}$ ( \cite[Form 64]{hr}): There are no amorphous sets.
\item $\mathbf{BPI}$ (the Boolean Prime Ideal Principle,  \cite[Form 14]{hr}): Every Boolean algebra has a prime ideal. (Equivalently: For every set $S$, every proper filter over $S$ can be extended to an ultrafilter over $S$.)
\item $\mathbf{M}(C, S)$: Every compact metrizable space is separable. (Cf. \cite{k02}, \cite{k01} and \cite{ktw0}.)
\item $\mathbf{M}(C, cuf)$: Every compact metrizable space has a cuf base.
\item $\mathbf{IFBI}$: Every infinite set is filterbase infinite. (Cf. \cite{kfbi}.)
\item $\mathbf{ISFBI}$: Every infinite set is strongly filterbase infinite.
\item $\mathbf{IDFBI}$: Every infinite set is dyadically filterbase infinite. 
\end{enumerate} 
\end{definition}

The forms $\mathbf{M}(C, cuf)$, $\mathbf{ISFBI}$ and $\mathbf{IDFBI}$ are newly introduced here. Basic facts about $\mathbf{M}(C, cuf)$  are given in Section \ref{s2}. That $\mathbf{ISFBI}$ and $\mathbf{IDFBI}$ are important in the theory of Hausdorff compactifications in $\mathbf{ZF}$ and independent of $\mathbf{ZF}$ is shown in Section \ref{s4}.

\subsection{Some known results}
\label{s1.3}

\begin{theorem}
\label{t:hm}
(Cf. \cite[Theorem 2.1]{hm0}.) $(\mathbf{ZFC})$. For every locally compact non-compact Hausdorff space $\mathbf{X}$, the following conditions (A)--(C) are all equivalent:
\begin{enumerate}
\item[(A)] There exist sequences $(\alpha_n\mathbf{X})_{n\in\mathbb{N}}$ of Hausdorff compactifications of $\mathbf{X}$ and $(\psi_n)_{n\in\mathbb{N}}$ of bijections $\psi_n:\{1,\ldots, 2^n\}\to\alpha_n X\setminus X$, such that, for every $n\in\mathbb{N}$, $\alpha_n\mathbf{X}\leq\alpha_{n+1}\mathbf{X}$ and if $h:\alpha_{n+1}\mathbf{X}\to\alpha_n\mathbf{X}$ is the continuous mapping with $h\circ\alpha_{n+1}=\alpha_n$, then $h^{-1}_{n+1}(\psi_{n+1}(2i-1))=h^{-1}_{n+1}(\psi(2i))=h^{-1}_n(\psi_n(i))$.
\item[(B)] Every non-empty compact metrizable space is a remainder of $\mathbf{X}$.
\item[(C)] There exists a sequence $(\mathcal{G}_n)_{n\in\mathbb{N}}$ such that, for every $n\in\mathbb{N}$, $\mathcal{G}_n$ is a family of pairwise disjoint open set of $\mathbf{X}$, $\mathcal{G}_n=\{G^n_i: i\in\{1,\ldots,2^n\}\}$ and, for every $i\in\{1,\ldots, 2^n\}$, the following conditions are satisfied:
\begin{enumerate}
\item[(i)] for every $i\in\{1,\ldots , 2^n\}$, $G^{n+1}_{2i-1}\cup G^{n+1}_{2i}\subseteq G^n_{i}$;
\item[(ii)] $K_n=X\setminus\bigcup\{G^n_i: i\in\{1,\ldots,2^n\}\}$ is compact;
\item[(iii)] for every $i\in\{1,\ldots, 2^n\}$, $K_n\cup G^n_i$ is non-compact.
\end{enumerate}
\end{enumerate}
\end{theorem}

\begin{corollary}
\label{c:hm}
(Cf. \cite[Corollary 3.1]{hm0}.)
$(\mathbf{ZFC})$
\begin{enumerate}
\item[(A)] If $\mathbf{X}$ is a locally compact Hausdorff space which admits a family $\{G_n: n\in\mathbb{N}\}$ of pairwise disjoint open sets such that the set $K=X\setminus\bigcup\limits_{n\in\mathbb{N}}G_n$ is compact and, for every $n\in\mathbb{N}$, $K\cup G_n$ is not compact, then every non-empty compact Hausdorff second-countable space is a remainder of $\mathbf{X}$.
\item[(B)] If $\mathbf{X}$ is the direct sum of a compact Hausdorff space and an infinite discrete space, then every non-empty compact Hausdorff second-countable space is a remainder of $\mathbf{X}$.
\end{enumerate}
\end{corollary}

\begin{theorem}
\label{t:Mag}
(Cf. \cite{mag} and \cite[Theorem 7.2]{ch}.) (Magill's Theorem.) $(\mathbf{ZF})$ Let $\mathbf{X}$ be locally compact Hausdorff space and let $\mathbf{K}$ be a compact Hausdorff space. Then $\mathbf{K}$ is a remainder of $\mathbf{X}$ if and only if $\mathbf{K}$ is a continuous image of a remainder of $\mathbf{X}$. Furthermore, if $\alpha\mathbf{X}$ is a Hausdorff compactification of $\mathbf{X}$ such that $\mathbf{K}$ is a continuous image of $\alpha X\setminus X$, then there exists a Hausdorff compactification $\gamma\mathbf{X}$ of $\mathbf{X}$ such that $\gamma\mathbf{X}\leq\alpha\mathbf{X}$ and $\mathbf{K}$ is homeomorphic to $\gamma X\setminus X$.
\end{theorem}

It was noticed in \cite{pw} and \cite{kw0} that the following Taimanov's Theorem is valid in $\mathbf{ZF}$:

\begin{theorem}
\label{t:Taim}
(Cf. \cite[Theorem 3.2.1]{En}.) (Taimanov's Theorem.) $(\mathbf{ZF})$ Let $\mathbf{X}$ be a dense subspace of a topological space $\mathbf{T}$ and let $f$ be a continuous mapping of $\mathbf{X}$ into a compact Hausdorff space $\mathbf{Y}$. Then $f$ is continuously extendable to a mapping $\tilde{f}:\mathbf{T}\to\mathbf{Y}$ if and only if, for each pair $A, B$ of disjoint closed sets of $\mathbf{Y}$, $\cl_{\mathbf{T}}(f^{-1}(A))\cap\cl_{\mathbf{T}}(f^{-1}(B))=\emptyset$.
\end{theorem}

 \begin{theorem} 
 \label{t:UMT}
(Cf. \cite[Corollary 4.8]{gt}.) (Urysohn's Metrization Theorem.) $(\mathbf{ZF})$ Every second-countable $T_3$-space is metrizable.
 \end{theorem}
 
 \begin{theorem}
 \label{s1t07} 
 (Cf. \cite{kt}.) $(\mathbf{ZF})$ A compact metrizable space $\mathbf{X}$ is second-countable if and only if it is separable which holds if and only if $\mathbf{X}$ is Loeb.
 \end{theorem}
 
 \begin{theorem}
 \label{s1t08}
 (Cf. \cite{k} and \cite{kt}.) $(\mathbf{ZF})$ 
 $$\mathbf{BPI}\rightarrow\mathbf{M}(C,S)\rightarrow\mathbf{CAC}_{fin}.$$
 \end{theorem}  
 
 \begin{theorem}
\label{t:loeb} 
(Cf. \cite{lo}.) $(\mathbf{ZF})$ Let $\kappa $ be an infinite cardinal number of von Neumann, $\{\mathbf{X}_{i}:i\in \kappa \}$ be a family of compact spaces, $\{f_{i}:i\in \kappa \}$ be a collection of functions such
that for every $i\in \kappa ,f_{i}$ is a Loeb function of $\mathbf{X}_{i}$. Then the Tychonoff product $%
\mathbf{X}=\prod\limits_{i\in \kappa }\mathbf{X}_{i}$ is compact. In particular, the Tychonoff cube $[0, 1]^{\omega}$ is compact.
\end{theorem}
 
 \begin{theorem}
 \label{s1t09}
 (Cf. \cite{ew}.)
$(\mathbf{ZF})$ Let $\mathbf{X}$ be a metrizable space which consists of at least two points. Then, for every non-empty set $J$, the space $\mathbf{X}^J$ is metrizable if and only if $J$ is a cuf set.
\end{theorem}

\begin{theorem}
\label{s1t016}
(Cf. \cite[Lemma 4.1]{kw2}.) $(\mathbf{ZF})$  Every non-empty second-countable compact Hausdorff space is a continuous image of the Cantor cube $\mathbf{2}^{\omega}$.
\end{theorem}

\begin{theorem}
\label{s1t017}
(Cf. \cite[Theorem 5.1]{kk}.) $(\mathbf{ZF})$ For every infinite set $J$, the Tychonoff cube $[0, 1]^J$ is a continuous image of the Cantor cube $\mathbf{2}^{\omega\times J}$. In consequence, if $\mathbf{2}^{\omega\times J}$ is compact (respectively, Loeb), then $[0, 1]^J$ is compact (respectively, Loeb).
\end{theorem}

\subsection{The content of the paper in brief}
\label{s1.4}

The main aim of this article is to investigate both Theorem \ref{t:hm} and Corollary \ref{c:hm} in the absence of the Axiom of Choice. The Axiom of Choice is involved in their proofs in \cite{hm0} and, besides, for a locally compact Hausdorff space satisfying conditions (A)--(C) of Theorem \ref{t:hm}, no direct construction of a Hausdorff compactification of $\mathbf{X}$ with the remainder homeomorphic to the Cantor cube $\mathbf{2}^{\omega}$ is given in \cite{hm0}. It is true in $\mathbf{ZFC}$ that a locally compact Hausdorff space $\mathbf{X}$ has a metrizable compactification if and only if $\mathbf{X}$ is second-countable. It is also true in $\mathbf{ZFC}$ that if $\mathbf{X}$ is a locally compact Hausdorff space which has a countable network, then every Hausdorff compactification $\gamma\mathbf{X}$  of $\mathbf{X}$ with a second-countable remainder is metrizable and second-countable because $\gamma\mathbf{X}$ has a countable network. The situation in $\mathbf{ZF}$ is different than in $\mathbf{ZFC}$. Namely, it was shown in \cite{keremTac} that there exists a model $\mathcal{M}$ of $\mathbf{ZF}$ in which there exists a countable compact Hausdorff space which is not metrizable, so not second-countable. Hence, in $\mathbf{ZF}$, a compact Hausdorff space with a countable network may fail to be metrizable. 

In Section \ref{s2}, we generalize Urysohn's Metrization Theorem by showing that it holds in $\mathbf{ZF}$ that every $T_3$-space having a cuf base is metrizable because it is embeddable in a metrizable Tychonoff cube (cf. Theorem \ref{s2t01}). We also show, among other results, that $\mathbf{M}(C, S)$ is equivalent to the conjunction $\mathbf{M}(C, cuf)\wedge\mathbf{CAC}_{fin}$, so a compact metrizable space can fail to have a cuf base in $\mathbf{ZF}$ (cf. Theorem \ref{s2t2}). We deduce that a metrizable weakly Loeb space has a cuf base if and only if it has a dense cuf set (cf. Corollary \ref{s2c017}).

In Section \ref{s3}, we prove in $\mathbf{ZF}$ that a space $\mathbf{X}$ is embeddable in a metrizable compact Tychonoff cube if and only if $\mathbf{X}$ is a second-countable $T_3$-space (cf. Theorem \ref{s3t2}). We show that there is a model of $\mathbf{ZF}$ in which there is a cuf set $J$ such that the Tychonoff cube $[0, 1]^J$ is not compact but the Cantor cube $\mathbf{2}^J$ is compact (cf. Theorem \ref{s3t3}). We extend Theorem \ref{s1t017} by showing in $\mathbf{ZF}$ that, for every infinite set $J$, if $[0, 1]^J$ is compact (respectively, Loeb), then $\mathbf{2}^{\omega\times j}$ is compact (respectively, Loeb) (cf. Theorem \ref{s3t5} and Corollary \ref{s3c6}).

Section \ref{s4} is devoted to Theorem \ref{t:hm} and Corollary \ref{c:hm}.  We prove in $\mathbf{ZF}$ that condition (C) of Theorem \ref{t:hm} is sufficient for a locally compact Hausdorff space to have every non-empty second-countable compact Hausdorff space as a remainder (cf. Theorem \ref{s4t:main2}). To do this in $\mathbf{ZF}$, given a locally compact Hausdorff space $\mathbf{X}$ satisfying (C) of Theorem \ref{t:hm}, we show a direct construction of a Hausdorff compactification $\gamma\mathbf{X}$ of $\mathbf{X}$ with $\gamma X\setminus X$ homeomorphic to the Cantor ternary set, and we show $\gamma\mathbf{X}$ is metrizable if $\mathbf{X}$ has a cuf base (cf. the proof to Theorem \ref{s4t:main2}). We also show that it is provable in $\mathbf{ZF}$ that, for every locally compact Hausdorff space $\mathbf{X}$, condition (A) of Theorem \ref{t:hm} is necessary and sufficient for $\mathbf{X}$ to have every non-empty second-countable compact Hausdorff space as a remainder (cf. Theorem \ref{s4t:main1}).  We modify condition (C) of Theorem \ref{t:hm} to get in $\mathbf{ZF}$  a necessary and sufficient condition for $\mathbf{X}$ to have every non-empty second-countable compact Hausdorff space as a remainder (cf. Theorem \ref{s4t:main1}). In particular, that a set $X$ is dyadically filterbase infinite is equivalent to our modification of (C) of Theorem \ref{t:hm} for the discrete space $\langle X, \mathcal{P}(X)\rangle$ (cf. Corollary \ref{s2c06}(i)). By describing another direct construction of a compactification, we show that Corollary \ref{c:hm}(A) is provable in $\mathbf{ZF}$ (cf. Theorem \ref{s2t02}) and we give some useful modifications of Corollary \ref{c:hm}(A) (cf., e.g., Theorem \ref{s2t03}). To get a topological characterization of a strongly filterbase infinite set (cf. Definition \ref{s1d:fbi}($a$)) in Corollary \ref{s4:def}, we give in $\mathbf{ZF}$ a necessary and sufficient condition for a locally compact Hausdorff space to have $\mathbb{N}(\infty)$ as a remainder (cf. Theorem \ref{t:char}) and, for an arbitrary locally compact Hausdorff space $\mathbf{X}$ satisfying our necessary condition to have $\mathbb{N}(\infty)$ as a remainder, we  show one more direct construction of a Hausdorff compactification $\gamma\mathbf{X}$ of $\mathbf{X}$ with the remainder $\gamma X\setminus X$ homeomorphic to $\mathbb{N}(\infty)$ (cf. the proof to Theorem \ref{t:char}). We also prove that if $\mathbf{D}$ is an amorphous discrete space, then, for every non-empty first-countable compact Hausdorff space, the space $\mathbf{K}\times\mathbf{D}(\infty)$ is the \v Cech-Stone compactification of $\mathbf{K}\times\mathbf{D}$ (cf. Theorem \ref{t:amdis}). In Section \ref{s5}, we use this result to a proof that the implications $(A)\rightarrow (B)$ and $(C)\rightarrow (B)$ of Theorem \ref{t:hm} are both independent of $\mathbf{ZF}$ (cf. Theorem \ref{s5t5.11}).

Section \ref{s5} contains other independence results relevant to Hausdorff compactifications. For instance, we notice that the statement ``All non-empty metrizable compact spaces are remainder of a metrizable compactification of $\mathbb{N}$ is independent of $\mathbf{ZF}$ because it implies $\mathbf{CAC}_{fin}$ in $\mathbf{ZF}$ (cf. Theorem \ref{s2t07}). We prove that the following implications are true in $\mathbf{ZF}$:  
$$\mathbf{IWDI}\rightarrow \text{Corollary } \ref{c:hm}(B)\rightarrow\mathbf{IDFBI}\rightarrow\mathbf{ISFBI}\rightarrow\mathbf{IFBI}\rightarrow\mathbf{NAS},$$
so condition (B) of Corollary \ref{c:hm} is independent of $\mathbf{ZF}$ (cf. Theorem \ref{s5t:main3}). We remark that, in view of Corollary \ref{s2c06}(iii), it holds in $\mathbf{ZF}$ that if a set $D$ is quasi Dedekind-infinite or equipotent to a subset of $\mathbb{R}$, then $D$ is dyadically filterbase infinite. In the proof to Theorem \ref{thm:1}, we show that there exists a model of $\mathbf{ZF}$ in which there is a quasi Dedekind-finite, dyadically filterbase infinite set which is not equipotent to $\mathbb{R}$. We prove that there exists a $\mathbf{ZF}$-model for $\mathbf{BPI}\wedge\neg\mathbf{IFBI}$ (cf. Theorem \ref{thm:2}) and deduce that $\mathbf{BPI}$ is independent of $\mathbf{NAS}\wedge\neg\mathbf{IFBI}$. We prove that the statement ``Every strongly filterbase infinite set is weakly Dedekind-infinite'' is independent of $\mathbf{ZF}$ (cf. Theorem \ref{thm:4}). To do this, in the proof to Theorem \ref{thm:4}, we construct a new permutation model in which there exists a strongly filterbase infinite set which is weakly Dedekind-finite. Finally, we notice that there exists a model of $\mathbf{ZF}$ in which there is a locally compact not completely regular Hausdorff space $\mathbf{X}$ such that all non-empty second-countable compact Hausdorff spaces are remainders of $\mathbf{X}$ (cf. Proposition \ref{s2p010}).

A list of open problems is included in Section \ref{s6}.

\section{A generalization of Urysohn's Metrization\newline Theorem}
\label{s2}

The following theorem generalizes Urysohn's Metrization Theorem (see \ref{t:UMT}):

\begin{theorem}
\label{s2t01}
$(\mathbf{ZF})$ If a $T_3$-space $\mathbf{X}$ has a cuf base, then $\mathbf{X}$ is metrizable.
\end{theorem}
\begin{proof}
We modify the proof to Proposition 4.6 in \cite{gt}. Let $\mathcal{B}=\bigcup_{n\in\omega}\mathcal{B}_n$ be a base of a $T_3$-space $\mathbf{X}$ such that, for every $n\in\omega$, $\mathcal{B}_n$ is a finite set. Let $H,K$ be a pair of non-empty disjoint closed sets of $\mathbf{X}$. For every $n\in\omega$, let $U_n=\bigcup\{ U\in\mathcal{B}_n: K\cap \text{cl}_{\mathbf{X}}U=\emptyset\}$ and $V_n=\{V\in\mathcal{B}_n: H\cap\text{cl}_{\mathbf{X}}V=\emptyset\}$. Let $D(H,K)=\langle U_H, V_K\rangle$ where $U_H=\bigcup_{n\in\omega}(U_n\setminus\text{cl}_{\mathbf{X}}(\bigcup_{j\in n+1}V_j))$ and 
$V_K=\bigcup_{n\in\omega}(V_n\setminus\text{cl}_{\mathbf{X}}(\bigcup_{j\in n+1}U_j))$. Let $J=\{\langle U, V\rangle: U,V\in\mathcal{B}\text{ and } \text{cl}_{\mathbf{X}}U\subseteq V\}$. Now, let us modify the proof to Corollary 4.7 in \cite{gt}. Namely, using the operator $D$ and arguing similarly to the proof of Urysohn's lemma, we can define in $\mathbf{ZF}$ a family $\mathcal{F}=\{f_{\langle U,V\rangle}: \langle U, V\rangle\}\in J\}$ of continuous functions $f_{\langle U, V\rangle}:\mathbf{X}\to [0,1]$ such that $U\subseteq f^{-1}_{\langle U, V\rangle}(0)$ and $X\setminus V \subseteq f^{-1}_{\langle U, V\rangle}(1)$. Then $\mathbf{X}$ is embeddable in $[0,1]^{\mathcal{F}}$. Since $\mathcal{F}$ is a cuf set, the cube $[0,1]^{\mathcal{F}}$ is metrizable by Theorem \ref{s1t09}. Hence $\mathbf{X}$ is metrizable.
\end{proof}

Let us show our first simple application of Theorem \ref{s2t01}. More applications of Theorem \ref{s2t01} are given in Sections \ref{s3} and \ref{s4}.

\begin{proposition}
\label{s2p02}
$(\mathbf{ZF})$
\begin{enumerate}
\item[(i)] If $\mathbf{X}$ is a non-compact locally compact Hausdorff space which has a cuf base, then the Alexandroff compactification $\mathbf{X}(\infty)$ of $\mathbf{X}$ is metrizable.
\item[(ii)]  (Cf. \cite[Proposition 3.5]{kw1}.) The Alexandroff compactification $\mathbf{D}(\infty)$ of an infinite discrete space is metrizable if and only if $\mathbf{D}$ is a cuf space. 
\item[(iii)] An infinite discrete space $\mathbf{D}$ has a metrizable compactification if and only if $\mathbf{D}$ is a cuf space. 
\end{enumerate}
\end{proposition}
\begin{proof}
(i) Assuming that $\mathcal{B}$ is a cuf base of a non-compact locally compact Hausdorff space $\mathbf{X}$, we fix an element $\infty\notin X$ and notice that if $\mathcal{B}^{\ast}=\{U\in\mathcal{B}: \text{cl}_{\mathbf{X}}(U)\text{ is compact}\}$, then the family $\mathcal{B}^{\ast}\cup\{X(\infty)\setminus \text{cl}_{\mathbf{X}}(U): U\in\mathcal{B}^{\ast}\}$ is a cuf base of $\mathbf{X}(\infty)$, so $\mathbf{X}(\infty)$ is metrizable by Theorem \ref{s2t01}. 

(ii) That (ii) holds was proved in \cite{kw1}. One can also notice that (ii) follows from (i) and the obvious fact that if $\mathbf{D}$ is an infinite discrete space and  $\infty$ has a countable base of neighborhoods in $\mathbf{D}(\infty)$, then $\mathbf{D}$ is a cuf space. 

(iii) Suppose that $\gamma\mathbf{D}$ is a metrizable compactification of an infinite discrete space $\mathbf{D}$. Let  $d$ be any metric which  induces the topology of $\gamma \mathbf{D}$ and let $\infty=\gamma D\setminus D$. We define a metric $\rho$ on $\mathbf{D}(\infty)$ as follows.  For $x,y\in D$, we put $\rho(x,y)=d(x,y)$.  For $x\in D$, we put $\rho(\infty, x)=\rho(x,\infty)=d(x, \gamma D\setminus D)$. Moreover, $\rho(\infty, \infty)=0$. Then $\rho$ induces the topology of $\mathbf{D}(\infty)$. Therefore, to conclude the proof to (iii), it suffices to apply (ii). 
\end{proof}

In view of Proposition \ref{s2p02}(ii), it is provable in $\mathbf{ZF}$ that, for every infinite discrete space $\mathbf{D}$, if $\mathbf{D}(\infty)$ is metrizable, then $\mathbf{D}(\infty)$ has a cuf base. However, the following theorem shows, among other facts, that the form $\mathbf{M}(C, cuf)$ is independent of $\mathbf{ZF}$ and, in every model of $\mathbf{ZF}+\neg\mathbf{M}(C, cuf)$, there exists a non-compact locally compact Hausdorff space $\mathbf{X}$ without any cuf base,  such that $\mathbf{X}(\infty)$ is metrizable. Hence, the converse of Proposition \ref{s2p02}(i) fails in a model of $\mathbf{ZF}$.

\begin{theorem}
\label{s2t2}
\begin{enumerate}
\item[(a)] The following conditions (i)--(iv) are satisfied in $\mathbf{ZF}$:
\begin{enumerate}
\item[(i)] $\mathbf{M}(C, cuf)$ is equivalent to the statement: for every non-compact locally compact Hausdorff space $\mathbf{X}$, if $\mathbf{X}(\infty)$ is metrizable, then $\mathbf{X}$ has a cuf base;
\item[(ii)] $\mathbf{M}(C, S)\leftrightarrow (\mathbf{M}(C, cuf)\wedge\mathbf{CAC}_{fin})$;
\item[(iii)] $\mathbf{M}(C, cuf)$ is equivalent to the statement: every compact metric space $\langle X, d\rangle$ has a base $\mathcal{B}$ such that $\mathcal{B}=\bigcup\limits_{n\in\mathbb{N}}\mathcal{B}_n$ where, for every $n\in\mathbb{N}$, $\mathcal{B}_n$ is a finite cover of $X$, $\mathcal{B}_{n+1}$ is a refinement of $\mathcal{B}_n$ and, for every $B\in\mathcal{B}_n$, $\delta_d(B)<\frac{1}{n}$;
\item[(iv)] $\mathbf{CMC}$ implies $\mathbf{M}(C, cuf)$.  
\end{enumerate}
\item[(b)] There exists a model of $\mathbf{ZF}$ in which $\mathbf{M}(C, cuf)$ fails. 
\item[(c)] In the Second Fraenkel Model $\mathcal{N}2$ of \cite{hr},  $\mathbf{M}(C, cuf)$ is true and $\mathbf{M}(C, S)$ is false.
\end{enumerate} 
\end{theorem}
\begin{proof}
($a$) To prove (i), let us consider an arbitrary infinite compact metrizable space $\mathbf{Y}$. Then $\mathbf{Y}$ has an accumulation point. Let $\infty$ be an accumulation point of $\mathbf{Y}$ and let $X=Y\setminus\{\infty\}$. Then the subspace $\mathbf{X}$ of $\mathbf{Y}$ is a locally compact, non-compact Hausdorff space such that $\mathbf{X}(\infty)=\mathbf{Y}$, so $\mathbf{X}(\infty)$ is metrizable. We notice that $\mathbf{X}(\infty)$ has a cuf base if and only if $\mathbf{Y}$ has a cuf base. Hence (i) is true in $\mathbf{ZF}$. 

To prove (ii), we notice that, by Theorem \ref{s1t08}, $\mathbf{M}(C, S)$ implies $\mathbf{CAC}_{fin}$. Clearly, $\mathbf{M}(C, S)$ implies that every compact metrizable space is second-countable. Hence $\mathbf{M}(C, S)$ implies that the conjunction $\mathbf{M}(C, cuf)\wedge\mathbf{CAC}_{fin}$ is true. 

To prove (iii), let us consider an arbitrary compact metric space $\mathbf{X}=\langle X, d\rangle$ and assume that $\mathcal{E}=\bigcup\limits_{n\in\mathbb{N}}\mathcal{E}_n$ is a base of $\mathbf{X}$ such that, for every $n\in\omega$, the family $\mathcal{E}_n$ is finite. Without loss of generality, we may assume that $\mathcal{E}_n\subseteq\mathcal{E}_{n+1}$ for every $n\in\mathbb{N}$. Since $\mathbf{X}$ is compact, we can define $m_1=\min\{ m\in\mathbb{N}: X=\bigcup\{E\in \mathcal{E}_m: \delta_d(E)<1\}$ and $\mathcal{B}_1=\{E\in\mathcal{E}_{m_	1}: \delta_d(E)<1\}$. Suppose that $k\in\mathbb{N}$ is such that $m_k$ has been defined so that the family $\mathcal{B}_k=\{E\in\mathcal{E}_{m_k}: \delta_d(E)<\frac{1}{k}\}$ covers $X$. Let 
$$\mathcal{V}_k=\{V\in\mathcal{E}: \delta_{d}(V)<\frac{1}{k+1}\wedge (\exists E\in\mathcal{B}_k) V\subseteq E\}.$$
Then $\mathcal{V}_k$ is an open cover of $\mathbf{X}$. By the compactness of $\mathbf{X}$, we can define
$m_{k+1}=\min\{m\in\mathbb{N}: X=\bigcup\{V\in\mathcal{V}_k: V\in\mathcal{E}_m\}\}$ and $\mathcal{B}_{k+1}=\mathcal{V}_k\cap\mathcal{E}_{m_{k+1}}$. In this way, we have inductively defined the sequence $(\mathcal{B}_k)_{k\in\mathbb{N}}$ witnessing that (iii) holds.

To prove (iv), let us consider an arbitrary compact metric space $\mathbf{X}=\langle X, d\rangle$. For every $n\in\mathbb{N}$, let $\mathcal{F}_n=\{ F\in[X]^{<\omega}: X=\bigcup\limits_{x\in F}B_d(x,\frac{1}{n})\}$. By the compactness of $\mathbf{X}$, for every $n\in\mathbb{N}$, the family $\mathcal{F}_n$ is non-empty. Assuming $\mathbf{CMC}$, we can fix a multiple choice function $f$ of the family $\{\mathcal{F}_n: n\in\mathbb{N}\}$. For every $n\in\mathbb{N}$, we define $\mathcal{B}_n=\{B_d(x, \frac{1}{n}): x\in\bigcup f(n)\}$ Then $\mathcal{B}=\bigcup\limits_{n\in\mathbb{N}}\mathcal{B}_n$ is a cuf base of $\mathbf{X}$. \smallskip

($b$) It was shown in \cite{ktw0} that there exists a model $\mathcal{M}$ of $\mathbf{ZF}$ in which the conjunction $\neg\mathbf{M}(C, S)\wedge\mathbf{CAC}_{fin}$ is true. Then it follows from (ii) that $\mathbf{M}(C, cuf)$ fails in $\mathcal{M}$.\smallskip

($d$) It is known that the conjunction  $\mathbf{CMC}\wedge\neg\mathbf{CAC}_{fin}$ is true in $\mathcal{N}2$ (see \cite[p. 126]{hr}). Hence, it follows from the proofs to (ii) and (iv) that $\mathbf{M}(C, S)$ fails in $\mathcal{N}2$ and $\mathbf{M}(C, cuf)$ is true in $\mathcal{N}2$.
\end{proof}

The known fact of $\mathbf{ZF}$ that every separable metrizable space is second-countable can be modified as follows:

\begin{proposition}
\label{s2p015}
$(\mathbf{ZF})$
If a metrizable space $\mathbf{X}$ has a dense cuf set, then $\mathbf{X}$ has a cuf base. In particular, every metrizable cuf space has a cuf base. 
\end{proposition}
\begin{proof}
Assume that $A=\bigcup\limits_{n\in\omega}A_n$ is a dense set in a metric space  $\mathbf{X}=\langle X, d\rangle$ such that, for every $n\in\omega$, $A_n$ is a non-empty finite set. For $n,m\in\omega$, we define $\mathcal{B}_{m,n}=\{B_d(x,\frac{1}{m+1}): x\in A_n\}$. One can easily observe that $\mathcal{B}=\bigcup\limits_{n,m\in\omega}\mathcal{B}_{n,m}$ is a cuf base of $\mathbf{X}$. 
\end{proof}

We do not know if it can be proved in $\mathbf{ZF}$ that every compact metrizable space which admits a cuf base has a dense cuf set. However, we can give below a partial solution to this problem.

\begin{proposition}
\label{s2p016}
$(\mathbf{ZF})$ Let $\mathbf{X}$ be a weakly Loeb regular space which has a cuf base. Then $\mathbf{X}$ has a dense cuf set. 
\end{proposition}

\begin{proof} Let $\mathcal{B}=\bigcup\limits_{n\in\omega}\mathcal{B}_n$ be a base of $\mathbf{X}$ such that, for every $n\in\omega$, the family $\mathcal{B}_n$ is finite. We may assume that $\emptyset\notin\mathcal{B}$. Let $f$ be a weak Loeb function of $\mathbf{X}$ and let $D_n=\bigcup\{f(\text{cl}_{\mathbf{X}}(U)): U\in\mathcal{B}_n\}$ for every $n\in\omega$. Then, for every $n\in\omega$, the set $D_n$ is finite. Since $\mathbf{X}$ is regular and $\mathcal{B}$ is a base of $\mathbf{X}$, it is easily seen that the set $D=\bigcup\limits_{n\in\omega}D_n$ is dense in $\mathbf{X}$.
\end{proof}

\begin{corollary}
\label{s2c017}
$(\mathbf{ZF})$ Let $\mathbf{X}$ be a metrizable weakly Loeb space. Then $\mathbf{X}$ has a cuf base if and only if $\mathbf{X}$ has a dense cuf set.
\end{corollary}

\section{A little more on metrizable Tychonoff and Cantor cubes}
\label{s3}

Let us recall that $\mathbf{BPI}$ is equivalent to each of the statements: ``every Tychonoff cube is compact'' and ``every Cantor cube is compact'' (see, e.g., \cite[Theorem 4.70]{her}). By Theorem \ref{s1t016}, it is provable in $\mathbf{ZF}$ that the Tychonoff cube $[0, 1]^{\omega}$ is a continuous image of the Cantor cube $\mathbf{2}^{\omega}$. By Theorem \ref{s1t09}, it holds in $\mathbf{ZF}$ that if $J$ is an infinite cuf set, then both cubes $[0, 1]^J$ and $\mathbf{2}^J$ are metrizable. One may ask if it is provable in $\mathbf{ZF}$ that, for every infinite cuf set $J$, the Tychonoff cube $[0, 1]^J$ is a continuous image of the Cantor cube $\mathbf{2}^J$. Bearing in mind Theorem \ref{s1t017}, one may also ask whether it can be proved in $\mathbf{ZF}$ that, for every infinite set $J$, if the Tychonoff cube $[0, 1]^J$ is compact (respectively, Loeb), then so is the Cantor cube $\mathbf{2}^{\omega\times J}$. In this section, we give answers to these questions. Moreover, by applying Theorem \ref{s2t01} and its proof, we obtain the following proposition and the forthcoming Theorem \ref{s3t2}.

\begin{proposition}
\label{s3p1}
$(\mathbf{ZF})$
\begin{enumerate}
\item[(i)] A Tychonoff cube is metrizable if and only if it has a cuf base.
\item[(ii)] A Cantor cube is metrizable if and only if it has a cuf base.
\item[(iii)] A $T_3$-space $\mathbf{X}$ has a cuf base if and only if $\mathbf{X}$ is embeddable in a metrizable Tychonoff cube.
\end{enumerate}
\end{proposition}
\begin{proof}
(i) Let $J$ be an infinite set. It follows from Theorem \ref{s2t01} that if the cube $[0, 1]^J$ has a cuf base, then it is metrizable. On the other hand, if $[0, 1]^J$ is metrizable, then, by Theorem \ref{s1t09}, $J$ is a cuf set. Suppose that $J=\bigcup\limits_{n\in\mathbb{N}}K_n$ where, for every $n\in\mathbb{N}$, $K_n$ is a non-empty finite set and $K_n\subseteq K_{n+1}$. For every $n\in\mathbb{N}$, let $\mathcal{E}_n=\{[0, \frac{1}{n}), (\frac{n-1}{n}, 1]\}\cup\{(\frac{i-1}{n}, \frac{i+1}{n}): i\in n\}$, $\mathcal{G}_n=\{\prod\limits_{k\in K_n}G_k: (\forall k\in K_n) (G_k\in\mathcal{E}_n)\}$ and $\mathcal{B}_n=\{\prod\limits_{j\in J}H_k: (\forall j\in J\setminus K_n)(H_k=[0,1]) \wedge \prod\limits_{k\in K_n}H_k\in\mathcal{G}_n\}$. Then $\mathcal{B}=\bigcup\limits_{n\in\mathbb{N}}\mathcal{B}_n$ is a cuf base of $[0, 1]^J$.
By a slight modification of the proof to (i), we can get a proof to (ii).

(iii) Let $\mathbf{X}$ be a $T_3$-space. We have shown in the proof to Theorem \ref{s2t01} that if $\mathbf{X}$ has a cuf base, then $\mathbf{X}$ is embeddable in a metrizable Tychonoff cube. On the other hand, if $\mathbf{X}$ is homeomorphic to a subspace of a metrizable Tychonoff cube $[0, 1]^J$ then $\mathbf{X}$ has a cuf base because, by (i), the cube $[0, 1]^J$ has a cuf base. 
\end{proof}

 The following theorem, taken together with Proposition \ref{s3p1} and the known fact that a cuf set may fail to be countable in $\mathbf{ZF}$, shows that it is unprovable in $\mathbf{ZF}$ that every $T_3$-space which has a cuf base can be embedded in a compact metrizable Tychonoff cube.

\begin{theorem}
\label{s3t2}
$(\mathbf{ZF})$ 
\begin{enumerate}
\item[(i)] For an infinite set $J$, the Tychonoff cube $[0, 1]^J$ is both compact and metrizable if and only if $J$ is countable. In particular, if $J$ is an uncountable cuf set, then the cube $[0, 1]^J$ is metrizable but not compact.
\item[(ii)] A $T_3$-space $\mathbf{X}$ is embeddable in a compact metrizable Tychonoff cube if and only if $\mathbf{X}$ is second-countable.
\end{enumerate}
\end{theorem}

\begin{proof}
(i) Let $J$ be an infinite set such that the cube $[0, 1]^J$ is metrizable. By Theorem \ref{s1t09}, $J$ is a cuf set. Theorem 6 of \cite{kttych} states that, for every infinite set $I$ and every infinite von Neumann ordinal $\alpha$, if the Tychonoff cube $[0, 1]^{I}$ is compact, then, for every family $\{A_i: i\in\alpha\}$ of non-empty finite subsets of $I$, the union $\bigcup\limits_{i\in\alpha}A_i$ is well-orderable. Hence, by \cite[Theorem 6]{kttych}, if $J$ is a cuf set and the cube $[0, 1]^J$ is compact, then $J$ is countable as a well-orderable union of finite sets. This, together with Theorem \ref{s1t09}, completes the proof to (i). 

(ii) Let $\mathbf{X}$ be a $T_3$-space, If $\mathbf{X}$ is embeddable in a compact metrizable Tychonoff cube $[0, 1]^{J}$, then, by (i), $\mathbf{X}$ is embeddable  in $[0, 1]^{\omega}$, so $\mathbf{X}$ is second-countable because $[0, 1]^{\omega}$ is second-countable. On the other hand, assuming that $\mathbf{X}$ is second-countable, the proof to Theorem \ref{s2t01} shows that there exists a countable set $\mathcal{F}$ such that $\mathbf{X}$ is embeddable in $[0, 1]^{\mathcal{F}\times\mathcal{F}}$. To complete the proof, it suffices to notice that, since $\mathcal{F}$ is countable, the cube $[0, 1]^{\mathcal{F}\times\mathcal{F}}$ is embeddable in the compact metrizable cube $[0, 1]^{\omega}$ (see Theorem \ref{t:loeb}).
\end{proof}

\begin{theorem}
\label{s3t3}
It is consistent with $\mathbf{ZF}$ the existence of an infinite cuf set $J$ such that $\mathbf{2}^J$ is compact but $[0, 1]^J$ is not compact. In particular, it is consistent with $\mathbf{ZF}$ the existence of an infinite cuf set $J$ such that $[0, 1]^J$ is not a continuous image of $\mathbf{2}^J$.
\end{theorem}
\begin{proof}
In \cite{howtach}, it has been established the existence of a $\mathbf{ZF}$-model $\mathcal{M}$ in which there exists an uncountable cuf set $J$ such that $\mathbf{2}^J$ is compact. Then, by Theorem \ref{s3t2}(i), the Tychonoff cube $[0, 1]^J$ is not compact in $\mathcal{M}$. Therefor,  it is true in $\mathcal{M}$ that  $[0, 1]^J$ is not a continuous image of the compact  Cantor cube $\mathbf{2}^{J}$. 
\end{proof}

We omit an easy and standard proof to the following proposition:
\begin{proposition}
\label{s3p3}
$(\mathbf{ZF})$ For every infinite set $J$, 
$$(\mathbf{2}^{\omega})^J\simeq\mathbf{2}^{\omega\times J}\simeq\mathbf{2}^{J\times\omega}\simeq(\mathbf{2}^J)^{\omega}.$$
\end{proposition}

\begin{remark}
\label{s3r4}
If $C$ is the Cantor ternary subset of $[0, 1]$, then we denote by $\mathbf{C}$ the space $\langle C,\tau(d_e\upharpoonright C\times C)\rangle$ where $d_e$ is the Euclidean metric on $\mathbb{R}$. 
\end{remark}

\begin{theorem}
\label{s3t5}
Let $J$ be an infinite set such that $[0,1]^{J}$ is compact (respectively, Loeb).
Then $\mathbf{2}^{\omega\times J}$ is compact (respectively, Loeb). 
\end{theorem}
\begin{proof}
As in Remark \ref{s3r4}, let $C$ be the Cantor ternary set. It is well known that there exists a homeomorphism from $\mathbf{2}^{\omega}$ onto $\mathbf{C}$. This implies that there exists a homeomorphic embedding $h:(\mathbf{2}^{\omega})^J\to [0, 1]^J$ such that  $h((2^{\omega})^J)=C^J$.  Clearly, $C^J$ is a closed subset of $[0, 1]^J$. Hence, if $[0, 1]^J$ is compact, so is  $(\mathbf{2}^{\omega})^J$, and if  $[0, 1]^J$ is Loeb, so is $(\mathbf{2}^{\omega})^J$. Proposition \ref{s3p3} completes the proof. 
\end{proof}

Theorems \ref{s1t017} and \ref{s3t5} lead to the following corollary:
 
\begin{corollary}
\label{s3c6}
$(\mathbf{ZF})$ For every infinite set $J$, the Tychonoff cube $[0,1]^J$ is compact (respectively, Loeb) if and only if the Cantor cube $\mathbf{2}^{\omega\times J}$ is compact (respectively, Loeb).
\end{corollary}

\section{Second-countable compact Hausdorff spaces as remainders}
\label{s4}

Let us pass to a deeper analysis of Theorem \ref{t:hm}. The proof to this theorem in \cite{hm0} (cf. \cite[Theorem 2.1]{hm0}) involves the Axiom of Choice and the existence of the \v Cech-Stone compactification, while, in $\mathbf{ZF}$, a Tychonoff space may fail to have its \v Cech-Stone compactification (cf, e.g., \cite{kw0}). In Section \ref{s5}, it is shown that Theorem \ref{t:hm} is unprovable in $\mathbf{ZF}$. However, we can offer the following $\mathbf{ZF}$-modification of Theorem \ref{t:hm}:

\begin{theorem}
\label{s4t:main1}
$(\mathbf{ZF})$. For every locally compact Hausdorff space $\mathbf{X}$, the following conditions are all equivalent:
\begin{enumerate}
\item[(a)] condition (A) of Theorem \ref{t:hm};
\item[(b)] every non-empty second-countable compact Hausdorff space is a remainder of $\mathbf{X}$;
\item[(c)] there exists a family $\mathcal{V}=\{\mathcal{V}^n_i: n\in\mathbb{N}, i\in \{1,\ldots, 2^n\}\}$ such that, for every $n\in\mathbb{N}$, the following conditions are satisfied:
\begin{enumerate}
\item[(i)] for every $i\in\{1,\ldots, 2^n\}$, $\mathcal{V}_i^n$ is a non-empty family of open sets of $\mathbf{X}$ such that $\mathcal{V}_i^n$ is stable under finite unions and finite intersections, and, for every $U\in V_i^n$, the set $\cl_{\mathbf{X}}U$ is non-compact;
\item[(ii)]  for every $i\in\{1,\ldots, 2^n\}$ and for any $U, V\in\mathcal{V}_i^n$, $\cl_{\mathbf{X}}(U)\setminus V$ is compact;
\item[(iii)] for every pair $i,j$ of distinct elements of $\{1,\ldots, 2^n\}$, for any $W\in \mathcal{V}_{i}^n$ and $G\in\mathcal{V}_j^n$,  there exist $U\in \mathcal{V}_{2i-1}^{n+1}, V\in \mathcal{V}_{2i}^{n+1}$ with $\cl_{\mathbf{X}}(U\cup V)\setminus W$ compact and $\cl_{\mathbf{X}}(( U\cup V)\cap G)$ compact; 
\item[(iv)] if, for every $i\in\{1,\ldots, 2^n\}$,   $V_i\in\mathcal{V}_i^n$, then $X\setminus\bigcup\limits_{i=1}^{2^n}V_i$ is compact.
\end{enumerate}
\end{enumerate}
\end{theorem}
\begin{proof} Let $\mathbf{X}$ be a locally compact space. We use the notation from Remark \ref{s3r4} for the Cantor ternary set $C$. As in \cite[the proof to Theorem 2.1]{hm0},  we can fix a sequence $(\mathcal{A}_n)_{n\in\mathbb{N}}$ of families $\mathcal{A}_n$ of pairwise disjoint clopen sets of $\mathbf{C}$ such that $\mathcal{A}=\bigcup\limits_{n\in\mathbb{N}}\bigcup\mathcal{A}_n$ is a base of $\mathbf{C}$ and, for every $n\in\mathbb{N}$ and $i\in\{1,\ldots, 2^n\}$, $C=\bigcup\limits_{i=1}^{2^n}A_i^n$,  $\mathcal{A}_n=\{A_i^n: i\in\{1,\ldots,2^n\}\}$  and $A_{2i-1}^{n+1}\cup A_{2i}^{n+1}=A_{i}^n$. We may assume that $C\cap X=\emptyset$.

$(b)\rightarrow (c)$  Suppose that  $\alpha\mathbf{X}$ is a Hausdorff compactification of  $\mathbf{X}$ such that the subspace $\alpha X\setminus X$ of $\alpha\mathbf{X}$ is equal to $\mathbf{C}$. Let $\tau_{\alpha}$ be the topology of $\alpha\mathbf{X}$. For arbitrary $n\in\mathbb{N}$ and $i\in\{1,\ldots, 2^n\}$, we define
$$\mathcal{V}_i^n=\{ V\in\tau: V\cup A_i^n\in\tau_{\alpha}\text{ and } \cl_{\alpha\mathbf{X}}(V)=\cl_{\mathbf{X}}(V)\cup A_i^n\}.$$

\noindent Let us check that the family $\mathcal{V}=\{\mathcal{V}_i^n: n\in\mathbb{N}, i\in\{1,\ldots, 2^n\}\}$ sastisfies conditions (i)-(iv) of ($c$). 

Let $n\in\mathbb{N}$ $i,j\in\omega$ and $i\neq j$. To see that $\mathcal{V}_i^n\neq\emptyset$, we take disjoint sets $H_1,H_2\in\tau_{\alpha}$ such that $C\setminus A_{i}^n\subseteq H_1$ and $A_{i}^n\subseteq H_2$. Then $H_2\cap X\in \mathcal{V}_i^n$. Let $U,V\in\mathcal{V}_i$. One can easily check that $U\cap V\in\mathcal{V}_i^n$ and $U\cup V\in\mathcal{V}_i^n$. Since $\cl_{\mathbf{X}}(U)\setminus V=\cl_{\alpha\mathbf{X}}(U)\setminus(V\cup A_i^n)$, the set $\cl_{\mathbf{X}}(U)\setminus V$ is compact. Let $W\in\mathcal{V}_i^n$ and $G\in\mathcal{V}_j^n$. There exist disjoint sets $G_1, G_2\in\tau_{\alpha}$ such that $G_1\cup G_2\subseteq H_2$, $A_{2i-1}^{n+1}\subseteq G_1$ and $A_{2i}^{n+1}\subseteq G_2$. Then $V_1=G_1\cap X\in\mathcal{V}_{2i-1}^{n+1}$, $V_2=G_2\cap X\in\mathcal{V}_{2i}^{n+1}$ and $V_1\cup V_2\in\mathcal{V}_i^n$.  Hence $\cl_{\mathbf{X}}(V_1\cup V_2)\setminus W$ is compact. There exists $G^{\prime}\in\mathcal{V}_j^n$ such that $G^{\prime}\subseteq H_1$. Then $(V_1\cup V_2)\cap G \subseteq \cl_{\mathbf{X}}(G)\setminus G^{\prime}$, so $\cl_{\mathbf{X}}( (V_1\cup V_2)\cap G)$ is compact.  All this taken together shows that $\mathcal{V}$ satisfies conditions (i)--(iii). It remains to show that $\mathcal{V}$ satisfies (iv). 

Suppose that $n\in\mathbb{N}$ and $V_i\in\mathcal{V}_i^n$ for every  $i\in\{1,\ldots, 2^n\}$. For every $i\in\{1,\ldots, 2^n\}$, we can fix $V_i^{\prime}\in\mathcal{V}_i^n$ such that  the set $K=X\setminus\bigcup\limits_{i=1}^{2^n}V_i^{\prime}$ is compact. Then $X\setminus\bigcup\limits_{i=1}^{2^n}V_i\subseteq K\cup\bigcup\limits_{i=1}^{2^n}(\cl_{\mathbf{X}}(V_i^{\prime}\setminus V_i)$, so $X\setminus\bigcup\limits_{i=1}^{2^n}V_i$ is compact. Hence $\mathcal{V}$ satisfies (iv). This completes the proof that ($b$) implies ($c$).

$(c)\rightarrow (b)$ Now, suppose that $\mathbf{X}$ has a family $\mathcal{V}=\{\mathcal{V}_i^n: n\in\mathbb{N}, i\in\{1,...,2^n\}\}$ satisfying conditions (i)-(iv) of ($c$). In the light of Theorems \ref{s1t016} and \ref{t:Mag}, to show that every non-empty second-countable compact Hausdorff space is a remainder of $X$, it suffices to prove that $\mathbf{C}$ is a remainder of $\mathbf{X}$. To this aim, we put  $Y=X\cup C$ and define
$$\mathcal{B}=\tau\cup\{(A_{i}^n\cup V)\setminus F: F\in\mathcal{K}(\mathbf{X}), n\in\mathbb{N}, i\in\{1,..., 2^n\}, V\in\mathcal{V}_i^n\}.$$
It follows from (i)-(iii) that $\mathcal{B}$ is a base for a Hausdorff topology $\tau_{Y}$ in $Y$ such that both $\mathbf{X}$ and $\mathbf{C}$ are subspaces of $\mathbf{Y}=\langle Y, \tau_Y\rangle$. Condition (i) implies that $X$ is dense in $\mathbf{Y}$. To check that $\mathbf{Y}$ is compact, we consider a family $\mathcal{G}\subseteq\mathcal{B}$ such that $Y=\bigcup\mathcal{G}$. By the compactness of $\mathbf{C}$, there exists a finite $\mathcal{G}_0\subseteq\mathcal{G}$ such that $C\subseteq\mathcal{G}_0$.  For every $G\in\mathcal{G}_0$, we choose $n(G)\in\mathbb{N}$, $i(G)\in\{1,\ldots, 2^n\}$, $V(G)\in\mathcal{V}_{i(G)}^{n(G)}$ and a compact subset $F(G)$ of $\mathbf{X}$, such that $G=(A_{i(G)}^{n(G)}\cup V(G))\setminus F(G)$. By the compactness of $F^{\ast}= (X\setminus \bigcup\limits_{G\in\mathcal{G}_0}V(G))\cup\bigcup\limits_{G\in\mathcal{G}_0}F(G)$, there exists a finite $\mathcal{G}_1\subseteq\mathcal{G}$ such that $F^{\ast}\subseteq \mathcal{G}_1$. Then $\mathcal{G}_0\cup\mathcal{G}_1$ is a finite subcover of $\mathcal{G}$. Hence $\mathbf{Y}$ is a Hausdorff compactification of $\mathbf{X}$ such that $Y\setminus X=C$. Hence ($c$) implies ($b$) and, in consequence, ($b$) and ($c$) are equivalent. 

$(c)\rightarrow (a)$.  Suppose that $\alpha\mathbf{X}$ is a Hausdorff compactification of $\mathbf{X}$ having $\mathbf{C}$ as the remainder. For every $n\in\mathbb{N}$, let $\alpha_{n}\mathbf{X}$ be the compactification of $\mathbf{X}$ obtained from $\alpha X$ by identifying the sets $A_{2i-1}^n$ and $A_{2i}^n$ with points and let $\psi_n:\{1,\ldots, 2^n\}\to\alpha_n X\setminus X$ be defined by: $\psi_{i}=A_i^n$ for every $i\in\{1,\ldots, 2^n\}$. Then $(\alpha_n\mathbf{X})_{n\in\mathbb{N}}$ and $(\psi_n)_{n\in\mathbb{N}}$ are sequences satisfying the requirements from condition (A) of Theorem \ref{t:hm}. Hence ($c$) implies ($a$).

$(a)\rightarrow (c)$ Let  $(\alpha_n\mathbf{X})_{n\in\mathbb{N}}$  be a sequence of Hausdorff compactifications of $\mathbf{X}$ and $(\psi_n)_{n\in\mathbb{N}}$ be a sequence of bijections $\psi_n:\{1,\ldots, 2^n\}\to\alpha_n X\setminus X$ such that, for every $n\in\mathbb{N}$, $\alpha_n\mathbf{X}\leq\alpha_{n+1}\mathbf{X}$ and if $h:\alpha_{n+1}\mathbf{X}\to\alpha_n\mathbf{X}$ is the continuous mapping with $h\circ\alpha_{n+1}=\alpha_n$, then $h^{-1}_{n+1}(\psi_{n+1}(2i-1))=h^{-1}_{n+1}(\psi(2i))=h^{-1}_n(\psi_n(i))$.  For every $n\in\mathbb{N}$ and $i\in\{1,\ldots 2^n\}$, let $\mathcal{V}_i^n$ be the collection of all $V\subseteq X$ such that $V$ is open in $\mathbf{X}$, $V\cup\{\psi_n(i)\}$ is open in $\alpha_n\mathbb{N}$ and $\cl_{\alpha_n\mathbf{X}}(V)=\cl_{\mathbf{X}}(V)\cup\{\psi_n(i)\}$. Using similar arguments to the ones  in the proof that ($b$) implies ($c$), one can check that $\mathcal{V}=\{\mathcal{V}_i: i\in\{1,\ldots, 2^n\}\}$ satisfies the conditions (i)-(iv) of ($c$). Hence ($a$) implies ($c$).
\end{proof}

The following corollary follows directly from Theorems \ref{s4t:main1}, \ref{t:Mag} and \ref{s1t016}.

\begin{corollary}
\label{s4c02}
$(\mathbf{ZF})$ $\mathbf{M}(C,S)$ implies that if a locally compact Hausdorff space $\mathbf{X}$ satisfies conditions ($a$)-($c$) of Theorem \ref{s4t:main1}, then every non-empty metrizable compact space is a remainder of $\mathbf{X}$.
\end{corollary}

Our next theorem shows that, in $\mathbf{ZF}$, condition (C) of Theorem \ref{t:hm} is sufficient for a locally compact Hausdorff space to have all non-empty second-countable compact Hausdorff spaces as remainders. Contrary to \cite{hm0}, to construct a Hausdorff compactification having the Cantor ternary set as a remainder, we need not refer to \v Cech-Stone compactifications nor any form of  $\mathbf{AC}$.

\begin{theorem}
\label{s4t:main2}
 $(\mathbf{ZF})$ Suppose that $\mathbf{X}$ is a locally compact Hausdorff space which admits a sequence $(\mathcal{G}_n)_{n\in\mathbb{N}}$ such that, for every $n\in\mathbb{N}$, $\mathcal{G}_n$ is a family of pairwise disjoint open sets of $\mathbf{X}$, $\mathcal{G}_n=\{G_i^n: i\in\{1,...,2^n\}\}$ and, for every $i\in\{1,..., 2^n\}$, the following conditions are satisfied:
\begin{enumerate}
\item[(i)] for every $i\in\{1,\ldots, 2^n\}$, $G^{n+1}_{2i-1}\cup G^{n+1}_{2i}\subseteq G^n_{i}$;
\item[(ii)] $K_n=X\setminus\bigcup\{G^n_i: i\in\{1,\ldots, 2^n\}\}$ is compact;
\item[(iii)] for every $i\in\{1,\ldots, 2^n\}$, $K_n\cup G^n_i$ is non-compact.
\end{enumerate}
Then every non-empty compact second-countable Hausdorff space is a remainder of $\mathbf{X}$. Furthermore, if $\mathbf{X}$ has a cuf base, then, for every non-empty compact second-countable Hausdorff space $\mathbf{Z}$, there exists a metrizable compactification $\gamma\mathbf{X}$ of $\mathbf{X}$ such that $\mathbf{Z}$ is homeomorphic to $\gamma X\setminus X$.
\end{theorem}

\begin{proof}
That every non-empty second-countable compact Hausdorff space is a remainder of $\mathbf{X}$ can be deduced from Theorem \ref{s4t:main1}; however, we need a direct less complicated construction of a compactification here. Therefore, as in the proof to Theorem \ref{s4t:main1}, for the Cantor ternary set $C$, we fix a sequence $(\mathcal{A}_n)_{n\in\mathbb{N}}$ of families $\mathcal{A}_n$ of pairwise disjoint clopen sets of $\mathbf{C}$ such that $\mathcal{A}=\bigcup\limits_{n\in\mathbb{N}}\bigcup\mathcal{A}_n$ is a base of $\mathbf{C}$ and, for every $n\in\mathbb{N}$ and $i\in\{1,\ldots, 2^n\}$,  $\mathcal{A}_n=\{A_i^n: i\in\{1,\ldots,2^n\}\}$, $C=\bigcup\limits_{i=1}^{2^n}A_i^n$  and $A_{2i-1}^{n+1}\cup A_{2i}^{n+1}=A_{i}^n$.  Assuming that $X\cap C=\emptyset$, we put $Y=X\cup C$ and define 
$$\mathcal{B}=\tau\cup\{(A_{i}^n\cup G_i^n)\setminus K: K\in\mathcal{K}(\mathbf{X}), n\in\mathbb{N}, i\in\{1,\ldots, 2^n\}\}.$$
One can easily verify that $\mathcal{B}$ is a base for a topology $\tau_Y$ in $Y$ such that space $\mathbf{Y}=\langle Y, \tau_Y\rangle$ is a compact Hausdorff space. Of course, $\mathbf{X}$ and $\mathbf{C}$ are subspaces of $\mathbf{Y}$ and, moreover, $X$ is dense in $\mathbf{Y}$. Let $\alpha\mathbf{X}=\mathbf{Y}$ and let $\mathbf{Z}$ be a non-empty second-countable compact Hausdorff space. By  Theorem \ref{t:Mag}, there exists a Hausdorff compactification $\gamma\mathbf{X}$ of $\mathbf{X}$ such that $\gamma\mathbf{X}\leq\alpha\mathbf{X}$ and $\gamma X\setminus X$ is homeomorphic to $\mathbf{Z}$.

Now, suppose that $\mathbf{X}$ has a cuf base $\mathcal{B}_X$. Let $\mathcal{B}_X=\bigcup\limits_{n\in\omega}\mathcal{B}_n$ where each $\mathcal{B}_n$ is finite. Let $\mathcal{B}_X^{\ast}=\{V\in\mathcal{B}_X: \text{cl}_{\mathbf{X}}(V)\in\mathcal{K}(\mathbf{X})\}$ and, for every $n\in\omega$, let $\mathcal{K}_n=\{\text{cl}_{\mathbf{X}}(V): V\in\mathcal{B}_X^{\ast}\cap\mathcal{B}_n\}$,   $K_n=\bigcup\mathcal{K}_n$ and $F_n=\bigcup\limits_{i\in n+1}K_i$.  We notice that the collection
$$\mathcal{B}^{\ast}=\mathcal{B}_X^{\ast}\cup\{ (A_i^n\cup G_i^n)\setminus F_j: j\in\omega, n\in\mathbb{N} \text{ and }i\in{1,\ldots ,2^n}\}$$
\noindent is a base of $\mathbf{Y}$ such that $\mathcal{B}^{\ast}$ is a cuf set. Then $\mathbf{Y}$ has a cuf base $\mathcal{B}^{\ast}_Y$ such that, for all $U,V\in\mathcal{B}^{\ast}_Y$, $U\cup V\in\mathcal{B}^{\ast}_Y$.  Since $\gamma\mathbf{X}\leq\alpha\mathbf{X}$, there exists a continuous mapping $f:\alpha\mathbf{X}\to\gamma\mathbf{X}$ such that $f\circ\alpha=\gamma$. Then the family $\mathcal{B}_{\gamma}=\{\gamma X\setminus f(\alpha X\setminus U): U\in\mathcal{B}^{\ast}_Y\}$ is a cuf base of $\gamma\mathbf{X}$. Hence $\gamma\mathbf{X}$ is metrizable by Theorem \ref{s2t01}. 
\end{proof}

Let us notice that it can be deduced from Theorem \ref{s4t:main2} that item (A) of Corollary \ref{c:hm} is provable in $\mathbf{ZF}$. However, for a non-empty second-countable compact Hausdorff space $\mathbf{Z}$, the construction of $\alpha\mathbf{X}$ with $\alpha X\setminus X$ homeomorphic to $\mathbf{Z}$, which is described in the proof to the forthcoming theorem, is simpler than that given in the proof to Theorem \ref{s4t:main2} and much simpler than that in \cite{hm0}.

\begin{theorem}
\label{s2t02}
$(\mathbf{ZF})$
Let $\mathbf{X}=\langle X, \tau\rangle$ be a locally compact Hausdorff space which has a family $\{G_n: n\in\omega\}$ of pairwise disjoint open sets such that the set $K=X\setminus\bigcup\limits_{n\in\omega}G_n$ is compact and, for every $n\in\omega$,  $K\cup G_n$ is non-compact. Then the following conditions are satisfied:
\begin{enumerate}
 \item[(i)] every non-empty second-countable compact Hausdorff space is a remainder of $\mathbf{X}$;
 \item[(ii)] if $\mathbf{X}$ has a cuf base, then, for every non-empty second-countable compact Hausdorff space $\mathbf{Z}$, there exists a metrizable compactification $\gamma\mathbf{X}$ of $\mathbf{X}$ such that $\gamma X\setminus X$ is homeomorphic to $\mathbf{Z}$. 
 \end{enumerate}
\end{theorem}
\begin{proof}
Let $\mathbf{Z}$ be a non-empty second-countable Hausdorff space. By Theorem \ref{t:UMT} (thus, also by Theorem \ref{s2t01}), $\mathbf{Z}$  is metrizable. In view of Theorem \ref{s1t07}, $\mathbf{Z}$ is separable. Let $D=\{z_n: n\in\omega\}$ be a countable dense set of $\mathbf{Z}$ and let $\mathcal{C}$ be a countable base of $\mathbf{Z}$. We may assume that $Z\cap X=\emptyset$. Let $Y=X\cup Z$.  For every $C\in\mathcal{C}$, let $U_C=C\cup\bigcup\{G_n: z_n\in C\}$ and let
$$\mathcal{B}=\tau\cup\{U_C\setminus F: C\in\mathcal{C}\text{ and } F\in\mathcal{K}(\mathbf{X})\}.$$
Then $\mathcal{B}$ is a base for a topology $\tau_Y$ in $Y$. Since $\mathbf{X}$ is a locally compact Hausdorff space, $G_m\cap G_n=\emptyset$  for every pair $m,n$ of distinct elements of $\omega$ and the space $\mathbf{Z}$ is also Hausdorff, it follows that the space $\mathbf{Y}=\langle Y, \tau_Y\rangle$ is Hausdorff. Clearly, $\mathbf{X}$ and $\mathbf{Z}$ are subspaces of $\mathbf{Y}$. Since the sets $G_n\cup K$ are all non-compact, $X$ is dense in $\mathbf{Y}$. One can easily check that $\mathbf{Y}$ is compact. 

Now, suppose that $\mathbf{X}$ has a cuf base $\mathcal{B}_X$. Let $\mathcal{B}_X=\bigcup\limits_{n\in\omega}\mathcal{B}_n$ where each $\mathcal{B}_n$ is finite. In much the same way, as in the proof to Theorem \ref{s4t:main2}, we notice that $\mathcal{B}_X^{\ast}=\{V\in\mathcal{B}_X: \text{cl}_{\mathbf{X}}(V)\in\mathcal{K}(\mathbf{X})\}$ is also a cuf base of $\mathbf{X}$. For every $n\in\omega$, we put $\mathcal{K}_n=\{\text{cl}_{\mathbf{X}}(V): V\in\mathcal{B}_X^{\ast}\cap\mathcal{B}_n\}$,   $K_n=\bigcup\mathcal{K}_n$ and $F_n=\bigcup\limits_{i\in n+1}K_i$. Then the collection
$$\mathcal{B}^{\ast}=\mathcal{B}_X^{\ast}\cup\{ U_C\setminus F_i: i\in\omega \text{ and }C\in\mathcal{C}\}$$
\noindent is a base of $\mathbf{Y}$ such that $\mathcal{B}^{\ast}$ is a cuf set. Hence $\mathbf{Y}$ is metrizable by Theorem \ref{s2t01}.
\end{proof}

\begin{corollary}
\label{s2c03}
$(\mathbf{ZF})$
 Let $\mathbf{X}_1$ and  $\mathbf{X}_2$ be disjoint locally compact Hausdorff spaces. Suppose that $\mathbf{X}_2$ admits a family $\{G_n: n\in\omega\}$ of pairwise disjoint open sets such that the set $K=X_2\setminus\bigcup\limits_{n\in\omega}G_n$ is compact and, for every $n\in\omega$,  $K\cup G_n$ is non-compact. Then, for the direct sum $\mathbf{X}=\mathbf{X}_1\oplus\mathbf{X}_2$, the following conditions are satisfied:
 \begin{enumerate}
\item[(i)] all non-empty second-countable compact Hausdorff spaces are remainders of $\mathbf{X}$;
\item[(ii)] if both $\mathbf{X}_1$ and $\mathbf{X}_2$ have cuf bases, then all non-empty second-countable compact Hausdorff spaces are remainders of metrizable compactifications of $\mathbf{X}$.
\end{enumerate}
\end{corollary}
\begin{proof}
If $\mathbf{X}_1$ is compact, then (i) and (ii) follow directly from Theorem \ref{s2t02}. Assuming that $\mathbf{X}_1$ is not compact, we fix an element $\infty\notin Y\cup X$ and consider the Alexandroff compactification $\mathbf{X_1}(\infty)$ of $\mathbf{X}_1$. Let $\mathbf{Z}$ be a non-empty second-countable compact Hausdorff space. It follows from Theorem \ref{s2t02}(i) that there exists a Hausdorff compactification $\gamma\mathbf{X}_2$ of $\mathbf{X}_2$ such that $\gamma X_2\setminus X_2$ is homeomorphic to $\mathbf{Z}$. We may assume that $\infty\notin\gamma X_2$. Then $\mathbf{X}_1(\infty)\oplus\gamma\mathbf{X}_2$ is a Hausdorff compactification of $\mathbf{X}$. Let us fix an element $z_0\in\gamma X_2\setminus X_2$ and denote by $\alpha\mathbf{X}$ the space obtained from $\mathbf{X}_1(\infty)\oplus\gamma\mathbf{X}_2$ by identifying the set $\{\infty, z_0\}$ with a point. Then $\alpha\mathbf{X}$ is a Hausdorff compactification of $\mathbf{X}$ such that $\alpha X\setminus X$ is homeomorphic to $\mathbf{Z}$. This shows that (i) holds. To prove (ii), let us notice that if both $\mathbf{X}_1$ and $\mathbf{X}_2$ have cuf bases, then it follows from Proposition \ref{s2p02} that $\mathbf{X}_1(\infty)$ has a cuf base and, moreover,  it follows from the proof to Theorem \ref{s2t02} that we may assume that $\gamma\mathbf{X}_2$  has a cuf base. Then $\mathbf{X}_1(\infty)\oplus\gamma\mathbf{X}_2$ has a cuf base. This implies that $\alpha{X}$ has a cuf base. Hence $\alpha\mathbf{X}$ is metrizable by Theorem \ref{s2t01}.    
\end{proof}

\begin{corollary}
\label{s2c003}
$(\mathbf{ZF})$ Let $\mathbf{X}$ be a locally compact Hausdorff space which admits a locally finite disjoint family $\{G_n: n\in\omega\}$ of non-compact clopen sets. Then all non-empty second-countable Hausdorff compact spaces are remainders of $\mathbf{X}$. Furthermore, if $\mathbf{X}$ has a cuf base, then all non-empty second-countable Hausdorff compact spaces are remainders of metrizable compactifications of $\mathbf{X}$.
\end{corollary}
\begin{proof}
Let $X_2=\bigcup\limits_{n\in\omega}G_n$ and $X_1=X\setminus X_2$. Since the family $\{G_n: n\in\omega\}$ is locally finite and consists of clopen sets, the set $X_2$ is closed. Hence,  the subspaces $\mathbf{X}_1$ and $\mathbf{X}_2$ of $\mathbf{X}$ are both open, so locally compact. Furthermore, $\mathbf{X}=\mathbf{X}_1\oplus\mathbf{X}_2$. Corollary \ref{s2c03} completes the proof. 
\end{proof}

From the proof to Corollary \ref{s2c03}, we can deduce that the following proposition holds:

\begin{proposition}
\label{s2p05}
$(\mathbf{ZF})$
Let $\mathbf{X}_1$ and $\mathbf{X}_2$ be disjoint locally compact Hausdorff spaces. If all non-empty second-countable compact Hausdorff spaces are remainders of $\mathbf{X}_2$ or of $\mathbf{X}_1$, then all non-empty second-countable compact Hausdorff spaces are remainders of $\mathbf{X}_1\oplus\mathbf{X}_2$. 
\end{proposition}

\begin{theorem}
\label{s2t03}
$(\mathbf{ZF})$ Let $\mathbf{X}$ be a non-compact locally compact $T_1$-space which has a cuf base of clopen sets. Then every non-empty second-countable compact Hausdorff space is the remainder of a metrizable compactification of $\mathbf{X}$. 
\end{theorem}
\begin{proof} 
Let $\mathcal{B}=\bigcup\limits_{n\in\omega}\mathcal{B}_n$ be a base of $\mathbf{X}$ such that, for every $n\in\omega$, the family $\mathcal{B}_n$ is finite and consists of clopen sets. Since $\mathbf{X}$ is non-compact, there exists a cover $\mathcal{V}$ of $X$ such that $\mathcal{V}\subseteq\mathcal{B}$ and $\mathcal{V}$ does not have a finite subcover. Let $V_n=\bigcup\{V\in\mathcal{V}: V\in\mathcal{B}_n\}$ for every $n\in\omega$. The sets $V_n$ are clopen and the cover $\{V_n: n\in\omega\}$ of $\mathbf{X}$ does not have a finite subcover. For every $n\in\omega$, let $W_{n+1}=V_{n+1}\setminus\bigcup\limits_{i\in n+1}V_i$ and $W_0=V_0$. The sets $W_n$ are all clopen and $X=\bigcup\limits_{n\in\omega} W_n$ but the cover $\{W_n: n\in\omega\}$ of $X$ does not have a finite subcover. Therefore, we may assume that $W_n\neq\emptyset$ for every $n\in\omega$. Let $\{A_k: k\in\omega\}$ be a partition of $\omega$ such that, for every $k\in\omega$, the set $A_k$ is infinite. For every $k\in\omega$, let us define $G_k=\bigcup\{W_n: n\in A_k\}$.  Then, for every $k\in\omega$, the set $G_k$ is open and non-compact in $\mathbf{X}$. Moreover, $X=\bigcup\limits_{k\in\omega}G_k$ and $G_i\cap G_j=\emptyset$ for each pair $i,j$ of distinct elements of $\omega$. To conclude the proof, it suffices to apply Theorem \ref{s2t02}.
\end{proof}

\begin{corollary} 
\label{s2c06}
$(\mathbf{ZF})$ For a set $D$, let  $\mathbf{D}=\langle D, \mathcal{P}(D)\rangle$. Then the following statements are true:
\begin{enumerate}
\item[(i)] $D$ is dyadically filterbase infinite if and only if $\mathbf{2}^{\omega}$ (equivalently, every non-empty second-countable compact Hausdorff space) is a remainder of $\mathbf{D}$.
\item[(ii)] If $\mathbf{D}$ is a cuf space, then every non-empty second-countable compact Hausdorff space is a remainder of a metrizable compactification of $\mathbf{D}$. In particular, all non-empty second countable compact Hausdorff spaces are remainders of metrizable compactifications of $\mathbb{N}$.
\item[(iii)] If $D$ is weakly Dedekind-infinite, then $D$ is dyadically filterbase infinite. In particular, if $D$ contains an infinite cuf set or $D$ is equipotent to a subset of $\mathbb{R}$, then $D$ is dyadically filterbase infinite. 
\end{enumerate}
\end{corollary}
\begin{proof}
We deduce from the proof to Theorem \ref{s4t:main1} that (i) is true. It follows directly from  Theorem \ref{s2t03} that (ii) holds. 

(iii) Assume that $D$ is weakly Dedekind-infinite. Then $\mathcal{P}(D)$ is Dedekind-infinite, so we may use \cite[Lemma 2.3]{knd} or similar arguments to that in the proof to Proposition 2 (b) in \cite{ktcell} to deduce that there exists a denumerable family of pairwise disjoint subsets of $D$. Suppose that $\mathcal{E}=\{E_n: n\in\omega\}$ is a family of pairwise disjoint non-empty subsets of $D$. Let $M=\{n\in\omega: E_n\text{ is infinite}\}$. If the set $M$ is infinite, we deduce from Corollary \ref{s2c003} that $D$ is dyadically filterbase infinite. Suppose that the set $M$ is finite. Then  $D$ contains an infinite cuf set $X_1=\bigcup\limits_{n\in\omega\setminus M}E_n$. It follows from (i) that all non-empty second-countable compact Hausdorff spaces are remainders of the subspace $\mathbf{X}_1$ of $\mathbf{D}$. For the subspace $\mathbf{X}_2$ of $\mathbf{D}$ with $X_2=D\setminus X_1$, we have $\mathbf{D}=\mathbf{X}_1\oplus\mathbf{X}_2$. Hence, by Proposition \ref{s2p05}, all non-empty second-countable compact Hausdorff spaces are remainders of $\mathbf{D}$. This, together with the fact that every infinite subset of $\mathbb{R}$ is weakly Dedekind-infinite, implies that the second statement of (iii) is also true.
\end{proof}

The following theorem gives in $\mathbf{ZF}$ a general internal characterization of locally compact Hausdorff spaces having $\mathbb{N}(\infty)$ as a remainder.

\begin{theorem}
\label{t:char}
$(\mathbf{ZF})$ For every locally compact Hausdorff space $\mathbf{X}=\langle X, \tau\rangle$, the following conditions ($a$)-($c$) are equivalent:
\begin{enumerate}
\item[(a)] there exist a sequence $(\alpha_n\mathbf{X})_{n\in\mathbb{N}}$ of Hausdorff compactifications of $\mathbf{X}$, a denumerable set $A$ and a bijection $\psi:\omega\to A$,  such that, for every $n\in\mathbb{N}$, $\alpha_n X\setminus X= \{\psi(i): i\in n\}$, $\alpha_n\mathbf{X}\leq\alpha_{n+1}\mathbf{X}$ and if $h_n:\alpha_{n+1}\mathbf{X}\to\alpha_n\mathbf{X}$ is the continuous mapping such that $h_n\circ\alpha_{n+1}=\alpha_n$, then $h(\psi(i))=\psi(i)$ for each $i\in n$;
\item[(b)]  $\mathbb{N}(\infty)$ is a remainder of $\mathbf{X}$;
\item[(c)]  there exists a family $\mathcal{V}=\{\mathcal{V}_i: i\in\omega\}$ which has the following properties:
\begin{enumerate}
\item[(i)] for every $i\in\omega$, $\mathcal{V}_i$ is a non-empty family of open sets of $\mathbf{X}$ such that $\mathcal{V}_i$ is stable under finite unions and finite intersections, and,  for every $U\in\mathcal{V}_i$, $\cl_\mathbf{X}(U)$ non-compact;

\item[(ii)] for every $i\in\omega$ and any $U,V\in\mathcal{V}_i$, the set $\cl_{\mathbf{X}}(U)\setminus V$ is compact;
\item[(iii)] for all $i,j\in\omega$ with $i\neq j$, there exist $V\in\mathcal{V}_i$ and $W\in\mathcal{V}_j$ such that $\cl_{\mathbf{X}}(V\cap W)$ is compact;

\item[(iv)] for every $n\in\omega$ and $V_i\in\mathcal{V}_i$ with $i\in n$, the set $X\setminus\bigcup\limits_{i\in n} V_i$ is non-compact.
\end{enumerate}
\end{enumerate} 
\end{theorem}

\begin{proof}
We may assume that $(\omega+1)\cap X=\emptyset$ and consider $\omega+1$ with its order topology. Then $\mathbb{N}(\infty)$ and $\omega+1$ are homeomorphic, so $\omega+1$ is a remainder of $\mathbf{X}$ if and only if $\mathbb{N}(\infty)$ is a remainder of $\mathbf{X}$.

$(c)\rightarrow (b)$ Let $\mathcal{V}=\{\mathcal{V}_i: i\in\omega\}$ be a family which satisfies conditions (i)--(iv) of ($c$). Let $Y=X\cup(\omega+1)$. For every $x\in X$, let $\mathcal{B}(x)=\{U\in \tau: x\in U\text{ and } \text{cl}_{\mathbf{X}}(U)\in\mathcal{K}(\mathbf{X})\}$. For $n\in\omega$, we put 
$$\mathcal{B}(n)=\{ \{n\}\cup (V\setminus K): V\in\mathcal{V}_n \text{ and } K\in\mathcal{K}(\mathbf{X})\}.$$ 

Moreover, let $\mathcal{B}(\omega)$ be the family of all sets of the form
$$((\omega+1)\setminus n)\cup(X\setminus\bigcup\limits_{i\in n}\text{cl}_{\mathbf{X}}(V_i)\setminus K)$$
\noindent where $K\in\mathcal{K}(\mathbf{X})$,  $n\in\omega$ and, for every $i\in n$, $ V_i\in\mathcal{V}_i$. We denote by $\tau_Y$ be the topology in $Y$ such that, for every $y\in Y$, the collection $\mathcal{B}(y)$ is a neighborhood base at $y$ in $\mathbf{Y}=\langle Y, \tau_Y\rangle$. It  follows from (ii), (iii) and the definition of $\tau_Y$ that the space $\mathbf{Y}$ is Hausdorff. Clearly, $\mathbf{X}$ and $\omega+1$ are both subspaces of $\mathbf{Y}$. That $X$ is  dense  in $\mathbf{Y}$ follows from (iv) and (ii). Let $\mathcal{B}=\bigcup\limits_{y\in Y}\mathcal{B}(y)$. Then $\mathcal{B}$ is a base of $\mathbf{Y}$.  To prove that $\mathbf{Y}$ is compact, let us suppose that $\mathcal{G}\subseteq\mathcal{B}$ and $Y=\bigcup\mathcal{G}$. There exist $n\in\omega$, $V_i\in\mathcal{V}_i$ for $i\in n$ and $K\in\mathcal{K}(\mathbf{X})$ such that $G=((\omega+1)\setminus n)\cup(X\setminus\bigcup\limits_{i\in n}\text{cl}_{\mathbf{X}}(V_i)\setminus K)\in\mathcal{G}$. For every $m\in n$, there exist $W_m\in\mathcal{V}_m$ and $K_m\in\mathcal{K}$ such that $G_m=\{m\}\cup (W_m\setminus K_m)\in \mathcal{G}$. Then the set  $F=X\setminus(G\cup\bigcup\limits_{m\in n}G_m)$ is a subset of $K\cup\bigcup\limits_{m\in n}K_m\cup\bigcup\limits_{m\in n}[\text{cl}_{\mathbf{X}}(V_m)\setminus W_m]$, so $F$ is compact. There exists a finite set $\mathcal{G}_0\subseteq\mathcal{G}$ such that $F\subseteq\bigcup\mathcal{G}_0$. Then $\mathcal{G}_0\cup\{G\}\cup\{G_m: m\in n\}$ is a finite subcover of $\mathcal{G}$. Hence $\mathbf{Y}$ is compact. This completes the proof that $\omega+1$ is a remainder of $\mathbf{X}$. 

$(b)\rightarrow (c)$ Now, suppose that $\omega+1$ is a remainder of $\mathbf{X}$. Let $\alpha\mathbf{X}$ be a Hausdorff compactification of $\mathbf{X}$ such that $\alpha X\setminus X=\omega+1$. Let $\tau_{\alpha}$ be the topology of $\alpha\mathbf{X}$. For every $i\in\omega$, we put 
$$\mathcal{V}_i=\{ V\in\tau: V\cup\{i\}\in\tau_{\alpha}\text{ and }\text{cl}_{\alpha\mathbf{X}}(V)=\text{cl}_{\mathbf{X}}(V)\cup\{i\}\}.$$
Let $\mathcal{V}=\{\mathcal{V}_i: i\in\omega\}$ and check that $\mathcal{V}$ satisfies conditions (i)--(iv).
 
Let $i,j\in\omega$ and $i\neq j$. To see that $\mathcal{V}_i\neq\emptyset$, we take disjoint sets $H,G\in\tau_{\alpha}$ such that $(\omega+1)\setminus\{i\}\subseteq H$ and $i\in G$. Then $G\cap X\in \mathcal{V}_i$. Let $U,V\in\mathcal{V}_i$. It is easily seen that $U\cap V\in\mathcal{V}_i$ and $U\cup V\in\mathcal{V}_i$. Since $\cl_{\mathbf{X}}(U)\setminus V=\cl_{\alpha\mathbf{X}}U\setminus(V\cup\{i\})$, the set $\cl_{\mathbf{X}}(U)\setminus V$ is compact. There exists $W\in\mathcal{V}_j$ such that $W\subseteq H$. Then $W\cap G=\emptyset$. All this taken together shows that $\mathcal{V}$ satisfies conditions (i)--(iii). It remains to show that $\mathcal{V}$ satisfies (iv).

Suppose that $n\in\omega$ and $V_i\in\mathcal{V}_i$ for $i\in n$ are such that the set $K=X\setminus\bigcup\limits_{i\in n}V_i$ is compact. Then there exist disjoint $H_0, G_0\in\tau_{\alpha}$ such that $(\omega+1)\setminus n\subseteq H_0$, $n\subseteq G_0$ and $K\subseteq X\setminus (H_0\cup G_0)$.  For every $i\in n$, there exists $U_i\in\mathcal{V}_i$ such that $U_i\subseteq G_0$. Then $H_0\cap X\subseteq \bigcup\limits_{i\in n}\cl_{\mathbf{X}}(U_i)\setminus K\subseteq \bigcup\limits_{i\in n}(\cl_{\mathbf{X}}(U_i)\setminus V_i)$, so $\cl_{\mathbf{X}}(H_0\cap X)$ is compact. This implies that $\cl_{\alpha\mathbf{X}}(H_0\cap X)\subseteq X$ but this is impossible. The contradiction obtained proves that $\mathcal{V}$ satisfies (iv).

$(b)\rightarrow (a)$ If $\alpha\mathbf{X}$ is a Hausdorff compactification of $\mathbf{X}$ with $\alpha X\setminus X=\omega+1$,  then we put $\alpha_n\mathbf{X}$ to be the compactification of $\mathbf{X}$ obtained from $\alpha X$ by identifying the set $(\omega+1)\setminus(n+1)$ with a point, $A=\omega+1$ and $\psi(i)=i$ for each $i\in\omega$.

$(a)\rightarrow (c)$  Assuming ($a$), we fix a sequence $(\alpha_n\mathbf{X})_{n\in\mathbb{N}}$ be a sequence of Hausdorff compactifications  of $\mathbf{X}$ and a bijection $\psi:\omega\to A$ witnessing that ($a$) is satisfied.
For every $n\in\mathbb{N}$, we denote by $\tau_n$ the topology of $\alpha_n X$ and define
$$\mathcal{V}_{n-1}=\{V\in\tau: V\cup\{\psi(n)\}\in\tau_{n+1}\text{ and } \cl_{\alpha_{n+1}\mathbf{X}}(V)=\cl_{\mathbf{X}}(V)\cup\{n\}\}.$$
Then, using similar arguments to the ones in the proof that ($b$) implies ($c$), one can check that $(\mathcal{V}_n)_{n\in\omega}$ satisfies conditions (i)-(iv) of ($c$). Hence, since ($c$) implies ($b$), $\mathbb{N}(\infty)$ is a remainder of $\mathbf{X}$.
\end{proof} 

As an immediate corollary to Theorem \ref{t:char}, we obtain the following topological characterization of strongly filterbase infinite sets introduced in Definition \ref{s1d:fbi}($a$). 

\begin{corollary} 
\label{s4:def}
$(\mathbf{ZF})$ 
 A set $A$ is strongly filterbase infinite if and only if $\omega+1$ is a remainder of the discrete space $\langle A, \mathcal{P}(A)\rangle$.
\end{corollary}

\begin{corollary}
\label{s4c13}
$(\mathbf{ZF})$ Every dyadically filterbase infinite set is strongly filterbase infinite, and every strongly filterbase infinite set is filterbase infinite. Every amorphous set is filterbase finite, so also strongly filterbase finite and dyadically filterbase finite.
\end{corollary}
\begin{proof}
In view of Corollaries \ref{s2c06}(i) and \ref{s4:def}, it is obvious that dyadically filterbase infinite sets are strongly filterbase infinite. If $\{\mathcal{V}_i: i\in\omega\}$ is a collection of filterbases in $A$ satisfying conditions (i)-(iv) of Definition \ref{s1d:fbi}($a$), then, by putting $\tilde{\mathcal{V}}_i=\{V\setminus F: V\in\mathcal{V}_i, F\in [A]^{<\omega}\}$ for every $i\in\omega$, we obtain a collection $\{\tilde{\mathcal{V}}_i: i\in\omega\}$ of free filterbases in $A$ which witnesses that $A$ is filterbase infinite. It is obvious that every amorphous set is filterbase finite, and this was noticed in \cite{kfbi}.  
\end{proof}

\begin{remark}
\label{s4r11}
(i) We do not know if the following statement is provable in $\mathbf{ZF}$ or independent of $\mathbf{ZF}$: For every locally compact Hausdorff space $\mathbf{X}$, if $\mathbf{X}$ has a denumerable remainder, then $\mathbb{N}(\infty)$ is a remainder of $\mathbf{X}$.

(ii)  One can prove that, in $\mathbf{ZFC}$, given a Hausdorff compactification $\alpha\mathbf{X}$ of a locally compact Hausdorff space $\mathbf{X}$ such that $\alpha X\setminus X$ is denumerable, there exists a sequence $(A_n)_{n\in\omega}$ of non-empty clopen sets of $\alpha X\setminus X$ such that $A_0=\alpha X\setminus X$ and, for every $n\in\omega$, $A_{n+1}$ is a proper subset of $A_n$. Then, to see that $\mathbb{N}(\infty)$ is a remainder of $\mathbf{X}$, we can check that condition ($a$) of Theorem \ref{t:char} is satisfied. Namely, we can define a sequence $(\alpha_n\mathbf{X})_{n\in\mathbb{N}}$ as follows. We denote by $\alpha_1\mathbf{X}$ the compactification of $\mathbf{X}$ obtained from $\alpha\mathbf{X}$ by identifying the set $A_0$ with a point. For every $n\in\omega$, we define $\alpha_{n+1}\mathbf{X}$ as the compactification of $\mathbf{X}$ obtained from $\alpha\mathbf{X}$ by identifying the sets $A_n$ and $A_i\setminus A_{i+1}$ for $i\in n$ with points. This also leads to a $\mathbf{ZFC}$-proof to \cite[Theorem 6.32]{ch}.
\end{remark}

It was noticed in \cite{kw0} that it holds in $\mathbf{ZF}$ that if $D$ is an amorphous set, then, for $\mathbf{D}=\langle D, \mathcal{P}(D)\rangle$, $\mathbf{D}(\infty)$ is the \v Cech-Stone compactification of $\mathbf{D}$, the spaces $\mathbf{D}(\infty)\times \mathbf{D}$ and $\mathbf{D}$ are both pseudocompact but $\mathbf{D}(\infty)\times\mathbf{D}(\infty)$ is not the \v Cech-Stone compactification of $\mathbf{D}(\infty)\times\mathbf{D}$ (see \cite[Theorem 2.32 and Corollary 2.27]{kw0}). Now, we can prove that, somewhat surprisingly, the following theorem holds in $\mathbf{ZF}$:

\begin{theorem}
\label{t:amdis}
$(\mathbf{ZF})$ Let $D$ be an amorphous set and let $\mathbf{K}$ be a non-empty first-countable compact Hausdorff space. Then, for $\mathbf{D}=\langle D, \mathcal{P}(D)\rangle$, $\mathbf{K}\times\mathbf{D}(\infty)$ is the \v Cech-Stone compactification of $\mathbf{K}\times\mathbf{D}$. 
\end{theorem}

\begin{proof}  Put $\mathbf{X}=\mathbf{K}\times\mathbf{D}$ and $\mathbf{Y}=\mathbf{K}\times\mathbf{D}(\infty)$.  Consider any pair $A,B$ of non-empty disjoint closed sets of $\mathbf{X}$. Suppose that $\cl_{\mathbf{Y}}(A)\cap\cl_{\mathbf{Y}}(B)\neq\emptyset$. There exists $x_0\in\mathbf{K}$ such that $\langle x_0, \infty\rangle\in\cl_{\mathbf{Y}}(A)\cap\cl_{\mathbf{Y}}(B)$. Let $\{G_n: n\in\mathbb{N}\}$ be a base of neighborhoods of $x_0$ in $\mathbf{K}$ such that $G_{n+1}\subseteq G_n$ for every $n\in\mathbb{N}$.  For any $n\in\mathbb{N}$, let $A_n=\{t\in D: (G_n\times\{t\})\cap A\neq\emptyset\}$ and $B_n=\{t\in D: (G_n\times\{t\})\cap B\neq\emptyset\}$. Since $D$ is amorphous, there exists $m\in\mathbb{N}$ such that $A_n=A_m$ and $B_n=B_m$ for every $n\in\mathbb{N}$ with $n\geq m$. We notice that if $t\in A_m\cap B_m$, then $\langle x_0, t\rangle\in \cl_{\mathbf{X}}(A)\cap\cl_{\mathbf{X}}(B)=A\cap B$. Hence, since $A\cap B=\emptyset$, the sets $A_m$ and  $B_m$ are disjoint. Since $D$ is amorphous, either $A_m$ or $B_m$ is finite.  Suppose that $A_m$ is finite. Let $V=D(\infty)\setminus A_m$ and $U=G_m\times V$. Then $U$ is a neighborhood of $\langle x_0, \infty\rangle$ in $\mathbf{Y}$, so $U\cap A\neq\emptyset$. There exist $z_0\in G_m$ and $y_0\in D\cap V$, such that $\langle z_0, y_0\rangle\in A$. Then $y_0\in A_m$ which is impossible because $y_0\in V$. The contradiction obtained proves that $\cl_{\mathbf{Y}}(A)\cap\cl_{\mathbf{Y}}(B)=\emptyset$. This, together with Theorem \ref{t:Taim}, implies that $\mathbf{K}\times\mathbf{D}(\infty)$ is the \v Cech-Stone compactification of $\mathbf{K}\times\mathbf{D}$. 
\end{proof}

\begin{remark}
\label{s4r15}
 Let $\mathcal{M}$ be any model of $\mathbf{ZF}$ in which there is an amorphous set. Let $\mathbf{D}$ be any discrete amorphous space in $\mathcal{M}$. Then, in view of Theorem \ref{t:amdis}, it holds in $\mathcal{M}$ that the space $\mathbf{X}=(\omega+1)\times \mathbf{D}$ is a metrizable pseudocompact, locally compact space which has $\mathbb{N}(\infty)$ as a remainder but $\mathbf{2}^{\omega}$ is not a remainder of $\mathbf{X}$. 
\end{remark}

\section{Independence results on compact metrizable remainders}
\label{s5}

In this section, to prove some of our independence results, we apply permutation models and the following transfer theorem which generalizes and extends Jech-Sochor Embedding Theorem (cf. \cite[Theorem 6.1]{Je}):

\begin{theorem}
\label{pin}
(The Pincus Transfer Theorem.) (Cf. \cite{pin1}, \cite{pin2}, \cite[p. 286]{hr}, \cite[Theorem 2.19]{ktw0}.) Let $\mathbf{\Psi}$ be a conjunction of statements that are either injectively
boundable or $\mathbf{BPI}$. If $\mathbf{\Psi}$ has a permutation model, then $\mathbf{\Psi}$ has a $\mathbf{ZF}$-model.
\end{theorem}

The notions of boundable and injectively boundable statements are given in \cite{pin1} and in \cite[Note 103, pp. 284--285]{hr}). Every boundable statement is equivalent to an injectively boundable statement (see \cite{pin1} or \cite[p. 285]{hr}). 

To show that Theorem \ref{t:hm} is unprovable in $\mathbf{ZF}$ and that it may happen in a model of $\mathbf{ZF}$ that a locally compact Hausdorff space $\mathbf{X}$ satisfies conditions ($a$)-($c$) of Theorem \ref{s4t:main1} but not all metrizable compact spaces are remainders of $\mathbf{X}$, we need the following lemma whose proof is included here for completeness:

\begin{lemma}
\label{s5l01}
$(\mathbf{ZF})$ If a Hausdorff space $\mathbf{Y}$ is a continuous image of a compact Hausdorff second-countable space, then $\mathbf{Y}$ is second-countable. Therefore, a non-empty Hausdorff space $\mathbf{Y}$ is a continuous image of the Cantor cube $\mathbf{2}^{\omega}$ if and only if $\mathbf{Y}$ is compact and second-countable.
\end{lemma}
\begin{proof}
Let $\mathbf{X}$ be a second-countable compact Hausdorff space and let $f:\mathbf{X}\to\mathbf{Y}$ be a continuous mapping of $\mathbf{X}$ onto a Hausdorff space $\mathbf{Y}$.  Suppose that $\mathcal{B}$ is a base of $\mathbf{X}$ such that $\mathcal{B}$ is stable under finite unions. Since the spaces $\mathbf{X}$ and $\mathbf{Y}$ are Hausdorff, it follows from the continuity of $f$ and the compactness of $\mathbf{X}$ that the collection $\mathcal{G}=\{Y\setminus f(X\setminus U): U\in\mathcal{B}\}$ is a base of $\mathbf{Y}$. If $\mathcal{B}$ is countable, so is $\mathcal{G}$. This, together with Theorem \ref{s1t016}, completes the proof.
\end{proof}

\begin{theorem}
\label{s5t5.11}
\begin{enumerate}
\item[(i)] $(\mathbf{ZF})$ Suppose that $\mathbf{D}$ is an amorphous discrete space and  $\mathbf{K}$ is a metrizable compact but not second-countable space. Then the space  $\mathbf{X}=\mathbf{2}^{\omega}\times\mathbf{D}$ satisfies the assumptions of Theorem \ref{s4t:main2}; however,  $\mathbf{K}$ is not a remainder of $\mathbf{X}$.
\item[(ii)] It is relatively consistent with $\mathbf{ZF}$ that there exists a locally compact Hausdorff space $\mathbf{X}$ such that all non-empty compact second-countable Hausdorff spaces are remainders of $\mathbf{X}$ but not all non-empty metrizable compact spaces are remainders of $\mathbf{X}$.
\item[(iii)] The implications $(A)\rightarrow (B)$ and $(C)\rightarrow (B)$ of Theorem \ref{t:hm} are both unprovable in $\mathbf{ZF}$.
\end{enumerate}
\end{theorem}
\begin{proof} (i) It follows from Theorem \ref{t:amdis} that $\mathbf{2}^{\omega}\times\mathbf{D}(\infty)$ is the \v Cech-Stone compactification $\beta\mathbf{X}$ of $\mathbf{X}$. Since $\beta X\setminus X$ is homeomorpbic to $\mathbf{2}^{\omega}$, while $\mathbf{K}$ is not second-countable, it follows from Lemma \ref{s5l01} that $\mathbf{K}$ is not a continuous image of $\beta X\setminus X$. This implies that $\mathbf{K}$ is not a remainder of $\mathbf{X}$. To see that $\mathbf{X}$ satisfies the assumptions of Theorem \ref{s4t:main2}, we fix a sequence $(\mathcal{A}_n)_{n\in\mathbb{N}}$ of families $\mathcal{A}_n$ of pairwise disjoint clopen sets of $\mathbf{2}^{\omega}$ such that $\mathcal{A}=\bigcup\limits_{n\in\mathbb{N}}\bigcup\mathcal{A}_n$ is a base of $\mathbf{2}^{\omega}$ and, for every $n\in\mathbb{N}$ and $i\in\{1,\ldots, 2^n\}$, $2^{\omega}=\bigcup\limits_{i=1}^{2^n}A_i^n$,  $\mathcal{A}_n=\{A_i^n: i\in\{1,\ldots, 2^n\}\}$  and $A_{2i-1}^{n+1}\cup A_{2i}^{n+1}=A_{i}^n$. For each $n\in\mathbb{N}$ and each $i\in\{1,\ldots, 2^n\}$, we define $G_i^n=A_i^n\times D$. Then, for $\mathcal{G}_n=\{G_i^n: i\in\{1,\ldots, 2^n\}\}$, the sequence $(\mathcal{G}_n)_{n\in\mathbb{N}}$ satisfies the requirements of Theorem \ref{s4t:main2}.

(ii) In Brunner's Model II (labeled as $\mathcal{N}43$ in \cite{hr}), $\mathbf{CAC}_{fin}$ and $\mathbf{NAS}$ are both false. Hence, by Theorem \ref{pin}, the conjunction $(\neg \mathbf{CAC}_{fin})\wedge\neg\mathbf{NAS}$ has a $\mathbf{ZF}$-model. Let $\mathcal{M}$ be any $\mathbf{ZF}$- model for $(\neg \mathbf{CAC}_{fin})\wedge\neg\mathbf{NAS}$. Let $\mathbf{D}$ be an amorphous discrete space in $\mathcal{M}$. Since $\mathbf{CAC}_{fin}$ fails in $\mathcal{M}$, there exists an uncountable cuf set $Y$ in $\mathcal{M}$ (cf. \cite[Form \textbf{[10 A]}]{hr}). It follows from Proposition \ref{s2p02}(i) that, in $\mathcal{M}$, the space $\mathbf{K}=\mathbf{Y}(\infty)$ is metrizable. Clearly, $\mathbf{K}$ is not second-countable because $Y$ is uncountable. In the light of (i), it is true in $\mathcal{M}$ that $\mathbf{K}$ is not a remainder of $\mathbf{2}^{\omega}\times \mathbf{D}$, while all non-empty second-countable Hausdorff compact spaces are remainders of $\mathbf{2}^{\omega}\times\mathbf{D}$ by Theorem \ref{s4t:main2}.

That (iii) holds follows from the proof to (ii), taken together with (i) and Theorem \ref{s4t:main1}.
\end{proof}

\begin{theorem}
\label{s2t07} $(\mathbf{ZF})$
The statement ``All non-empty metrizable compact spaces are remainders of metrizable compactifications of $\mathbb{N}$'' is equivalent to $\mathbf{M}(C, S)$ and, thus, it implies $\mathbf{CAC}_{fin}$.
\end{theorem}
\begin{proof}  
By Theorem \ref{s1t08}, $\mathbf{M}(C, S)$ implies $\mathbf{CAC}_{fin}$. It follows from Corollary \ref{s2c06}(i)  that $\mathbf{M}(C, S)$ implies that every non-empty compact metrizable space is the remainder of a metrizable compactification of $\mathbb{N}$. 

To complete the proof, let us suppose that $\mathbf{Z}$ is a non-empty compact metrizable space which is the remainder of a metrizable compactification $\gamma\mathbb{N}$ of $\mathbb{N}$. Since $\gamma\mathbb{N}$ is a separable metrizable space, it is second-countable. Hence $\mathbf{Z}$ is second-countable, so $\mathbf{Z}$ is separable by Theorem \ref{s1t07}.
\end{proof}

\begin{corollary}
\label{s3c06}
The statement ``All non-empty metrizable compact spaces are remainders of metrizable compactifications of $\mathbb{N}$'' is independent of $\mathbf{ZF}$.
\end{corollary}

\begin{proposition}
\label{s5p06}
\begin{enumerate}
\item[(i)] $(\mathbf{ZF})$ $\mathbf{CAC}$ implies that every compact Hausdorff space with a countable network is second-countable.
\item[(ii)] It is relatively consistent with $\mathbf{ZF}$ the existence of a compact Hausdorff, not second-countable space having a countable network. 
\end{enumerate}
\end{proposition}
\begin{proof}
(i) Let $\mathbf{X}=\langle X, \tau\rangle$ be a  compact Hausdorff space which has a countable network $\mathcal{N}$. Assume that $X$ consists of at least two points.  Let $\mathcal{S}=\{\langle A, B\rangle\in\mathcal{N}\times\mathcal{N}: \cl_{\mathbf{X}}(A)\cap \cl_{\mathbf{X}}(B)=\emptyset\}$. For every $\langle A, B\rangle\in\mathcal{S}$, let $\mathcal{S}(\langle A, B\rangle)=\{\langle U, V\rangle\in\tau\times\tau: \cl_{\mathbf{X}}(A)\subseteq U, \cl_{\mathbf{X}}(B)\subseteq V\text{ and } U\cap V=\emptyset\}$. Since $\mathbf{X}$ is a compact Hausdorff space which consists of at least two points, it follows that $\mathcal{S}\neq\emptyset\neq \mathcal{S}(\langle A, B\rangle)$ for every $\langle A, B\rangle\in\mathcal{S}$. By $\mathbf{CAC}$, there exists $f\in\prod\{\mathcal{S}(\langle A, B\rangle): \langle A, B\rangle\in\mathcal{S}\}$. Now, we can mimic the proof to \cite[Theorem 3.1.19]{En} to show that $\mathbf{X}$ has a countable base. Namely,  let $\mathcal{U}=\{U\in \tau: (\exists V\in\tau)(\exists \langle A, B\rangle\in\mathcal{S}) \langle U, V\rangle=f(\langle A, B\rangle)\}$  and $\mathcal{V}= \{V\in \tau: (\exists U\in\tau)(\exists \langle A, B\rangle\in\mathcal{S}) \langle U, V\rangle=f(\langle A, B\rangle)\}$. The family $\mathcal{B}$ of all finite intersections of members of $\mathcal{U}\cup\mathcal{V}$ is a countable base of some Hausdorff topology $\tau^{\ast}$ in $X$ such that $\tau^{\ast}\subseteq\tau$. Then $\tau^{\ast}=\tau$ because $\mathbf{X}$ is compact and $\langle X, \tau^{\ast}\rangle$ is Hausdorff. 

(ii) It was shown in \cite{keremTac} that there exists a model $\mathcal{M}$ of $\mathbf{ZF}$ in which there exists a compact Hausdorff countable space which is not metrizable, so not second-countable. Clearly, every countable space has a countable network. Hence (ii) holds.
\end{proof}

\begin{remark}
\label{s5r07}
(i)  Let $\mathbf{X}$ be a locally compact Hausdorff space which has a cuf base. Suppose that $\gamma\mathbf{X}$ is a Hausdorff compactification of $\mathbf{X}$ such that the remainder $\gamma X\setminus X$ is compact and second-countable. We do not know if $\gamma\mathbf{X}$ must be metrizable in $\mathbf{ZF}$. However,  one can easily show in $\mathbf{ZF}$ that $\gamma\mathbf{X}$ has a network $\mathcal{N}$ such that $\mathcal{N}$ is a cuf set; therefore, in view of Proposition \ref{s5p06} and Theorem \ref{t:UMT}, $\mathbf{CAC}$ implies that $\gamma\mathbf{X}$ is metrizable.  Using similar arguments to the ones from the proof to Theorem \ref{s2t03}, one can show that it holds in $\mathbf{ZF}$ that if a locally compact non-compact Hausdorff space has a cuf base, then all Hausdorff compactifications of $\mathbf{X}$ with finite remainders are metrizable. 

(ii) Using similar arguments to the ones from the proof to Theorem \ref{s2t02}, one can show that it holds in $\mathbf{ZF}$ that if a locally compact non-compact Hausdorff space has a cuf base, then all Hausdorff compactifications of $\mathbf{X}$ with finite remainders are metrizable. 
 \end{remark}

Condition ($a$) of our next theorem shows that the statements $\mathbf{ISFBI}$ and $\mathbf{IDFBI}$ are both independent of $\mathbf{ZF}$, while condition ($c$), together with Theorem \ref{thm:1} leads to the conclusion that the converse of the second implication in Corollary \ref{s2c06}(iii) is not true in $\mathbf{ZF}$ because it fails in a model of $\mathbf{ZF}$.

\begin{theorem}
\label{s5t:main3}
\begin{enumerate}
\item[(a)] The following implications are true in $\mathbf{ZF}$ (and also in $\mathbf{ZFA}$):
$$\mathbf{IWDI}\rightarrow\mathbf{IDFBI}\rightarrow\mathbf{ISFBI}\rightarrow\mathbf{IFBI}\rightarrow\mathbf{NAS}.$$

\item[(b)]  Condition (B) of Corollary \ref{c:hm} is unprovable in $\mathbf{ZF}$.

\item[(c)] Let  $\mathcal{M}$ be a model of $\mathbf{ZF}$ in which there indeed exists a disjoint denumerable family of amorphous sets without a partial multiple choice function.  Then it is true in $\mathcal{M}$ that there exists a dyadically filterbase infinite set $A$ which neither contains an infinite cuf set nor is equipotent to a subset of $\mathbb{R}$.
\end{enumerate}
\end{theorem}

\begin{proof}
($a$) This follows directly from Corollaries \ref{s2c06}(iii) and \ref{s4c13}. For a proof to ($b)$, we notice that, in view of ($a$), Condition (B) of Corollary \ref{c:hm} is false in every model of $\mathbf{ZF}+\neg\mathbf{NAS}$, for instance, in the model $\mathcal{M}37$ in \cite{hr}.

($c$) Let $\mathcal{A}=\{A_n: n\in\omega\}$ be a disjoint family of amorphous sets in a model $\mathcal{M}$ of $\mathbf{ZF}$ such that $\mathcal{A}$ does not have a partial multiple choice function in $\mathcal{M}$. Then it is true in $\mathcal{M}$ that the set $A=\bigcup\limits_{n\in\omega}A_n$ is quasi Dedekind-finite. Since $A$ is not a cuf set in $\mathcal{M}$ and  every amorphous subset of $\mathbb{R}$ is finite, it is true in $\mathcal{M}$ that there is no injection $f: A\to\mathbb{R}$. By Corollary \ref{s2c06}(iii), $A$ is dyadically filterbase infinite in $\mathcal{M}$ because $A$ is weakly Dedekind-infinite in $\mathcal{M}$. 
\end{proof}

In the following theorem, we show that there exists a model $\mathcal{M}$ of $\mathbf{ZF}$ which satisfies the assumptions of ($c$) of Theorem \ref{s5t:main3}. 

\begin{theorem}
\label{thm:1}
It is relatively consistent with $\mathbf{ZF}$ that there exists a denumerable collection of amorphous sets with no partial multiple choice function. 
\end{theorem}
\begin{proof}

Let $\mathbf{\Psi}$ be the statement: There exists a denumerable family $Y$ of amorphous sets which does not have a partial multiple choice function. To show that $\mathbf{\Psi}$ has a permutation model, let us consider the Brunner-Pincus Model $\mathcal{N}26$ in \cite{hr}. We recall that $\mathcal{N}26$ is a permutation model which can be defined as follows. We start with a ground model $M$ of $\mathbf{ZFA}+\mathbf{AC}$ which has  the set $A$ of all atoms such that, in $M$, $A=\bigcup\limits_{n\in\omega}A_n$ where, for all $m,n\in\omega$,  $A_n$ is infinite and $A_m\cap A_n=\emptyset$. Let $\mathcal{G}$ be the group of all permutations $\phi$ of $A$ such that, for every $n\in\omega$, $\phi(A_n)=A_n$. Let $\mathcal{I}=[A]^{<\omega}$. In the terminology established in \cite[Definition 2.9]{kw0} (see also \cite[Chapter 4]{Je}),  $\mathcal{N}26$ is the permutation model determined by $M$, $\mathcal{G}$ and the ideal of supports $\mathcal{I}$. Then, for every $n\in\omega$, the set $A_n$ is amorphous in $\mathcal{N}26$ (see \cite[p. 202]{hr}). It is easily seen that the family $\{A_n: n\in\omega\}$ does not have a partial multiple choice function in $\mathcal{N}26$. Hence $\mathbf{\Psi}$ is satisfied in $\mathcal{N}26$.

To transfer  $\mathbf{\Psi}$ to a $\mathbf{ZF}$-model, we will apply the Jech--Sochor First Embedding Theorem (see \cite[Theorem 6.1]{Je})
which embeds an arbitrarily large initial segment of a given permutation model
of $\mathbf{ZFA}$ into a symmetric model of $\mathbf{ZF}$. That is, suppose that $\mathcal{N}$ is a permutation model of $\mathbf{ZFA}$ with a set $A$ of atoms. Let $\alpha$ be an ordinal of the groud model of $\mathcal{N}$. Then, by \cite[Theorem 6.1]{Je}, there exists a symmetric model $\mathcal{M}$ of $\mathbf{ZF}$ and an $\in$-embedding $x\mapsto\tilde{x}$ of
$\mathcal{N}$ into $\mathcal{M}$ such that $(\mathcal{P}^{\alpha}(A))^{\mathcal{N}}$ is $\in$-isomorphic to $(\mathcal{P}^{\alpha}(\tilde{A}))^{\mathcal{M}}$,
where $\tilde{A}=\{\tilde{a}:a\in A\}$ (see \cite[equation (8.1), p. 128]{Je} for the definition of $\tilde{a}$). 

An application of the Jech--Sochor theorem concerns existential statements of the kind $\exists Y\varphi(Y,\gamma)$
where $\varphi(Y,\gamma)$ is a formula such that the only quantifiers in $\varphi$ are $\exists u\in \mathcal{P}^{\gamma}(Y)$
and $\forall u\in \mathcal{P}^{\gamma}(Y)$ (see \cite[Problem 1, p. 94]{Je}). More specifically, assume that $\mathcal{N}$ is a permutation model of $\mathbf{ZFA}$
satisfying $\exists Y\varphi(Y,\gamma)$. Let $Y$ be a set of $\mathcal{N}$ such that $\varphi(Y,\gamma)$ is true in $\mathcal{N}$. Suppose that $\alpha$ is an ordinal such that $\mathcal{P}^{\gamma}(Y)\subseteq\mathcal{P}^{\alpha}(A)$ where $A$ is the set of atoms of $\mathcal{N}$. By the embedding theorem, there exists a symmetric model
$\mathcal{M}$ of $\mathbf{ZF}$ such that $(\mathcal{P}^{\alpha}(A))^{\mathcal{N}}$ is $\in$-isomorphic to $(\mathcal{P}^{\alpha}(\tilde{A}))^{\mathcal{M}}$.
Since the quantifiers in $\varphi$ are restricted to $\mathcal{P}^{\gamma}(Y)\subseteq\mathcal{P}^{\alpha}(A)$,
it follows that $\mathcal{M}$ satisfies $\varphi(\tilde{Y},\gamma)$, and therefore $\exists Y\varphi(Y,\gamma)$ is true in $\mathcal{M}$. 

Now, let us come back to the model $\mathcal{N}26$. We fix any ordinal $\gamma$ with $\omega<\gamma<\omega_1$ in the ground model $M$ of $\mathbf{ZFA+AC}$. Let $\varphi(Y,\gamma)$ stand for the formula: $Y$ is a denumerable set of pairwise disjoint amorphous sets such that $Y$ does not have a partial multiple choice function. Clearly, the  existential statement $\exists Y\varphi(Y,\gamma)$, described above, is such that all quantifiers in $\varphi$ are of the kind $\exists u\in \mathcal{P}^{\gamma}(Y)$ or $\forall u\in \mathcal{P}^{\gamma}(Y)$.

In particular, we know that $\exists Y\varphi(Y,\gamma)$ is true in $\mathcal{N}26$ by taking $Y$ to be the denumerable partition $\{A_{n}:n\in\omega\}$ of the set $A$ of atoms of $\mathcal{N}26$ into the amorphous sets $A_{n}$ ($n\in\omega$). Letting $\kappa$ be a regular cardinal such that $\kappa>|\mathcal{P}^{\gamma}(A)|$ in $M$, one can follow the proof of \cite[Theorem 6.1]{Je} in order to obtain a symmetric model $\mathcal{M}$ and a set $\tilde{A}=\{\tilde{a}:a\in A\}\in\mathcal{M}$ such that $(\mathcal{P}^{\gamma}(A))^{\mathcal{N}26}$ is $\in$-isomorphic to $(\mathcal{P}^{\gamma}(\tilde{A}))^{\mathcal{M}}$.
 
Let $\tilde{Y}=\{\tilde{A_{n}}:n\in\omega\}$, where $\tilde{A_{n}}=\{\tilde{a}:a\in A_{n}\}$. Then $\tilde{Y}\in(\mathcal{P}^{2}(\tilde{A}))^{\mathcal{M}}$, so $\tilde{Y}\in(\mathcal{P}^{\gamma}(\tilde{A}))^{\mathcal{M}}$.  

For every $n\in\omega$, $\tilde{A_{n}}$ is amorphous in the model $\mathcal{M}$. Indeed, let $n\in\omega$. Assume, by way of contradiction, that $\tilde{A_{n}}$ has a partition $\mathcal{U}$ into two infinite sets in $\mathcal{M}$. Now $\mathcal{U}\in(\mathcal{P}^{2}(\tilde{A_{n}}))^{\mathcal{M}}\subseteq(\mathcal{P}^{2}(\tilde{A}))^{\mathcal{M}}$. Thus $\mathcal{U}\in(\mathcal{P}^{\gamma}(\tilde{A}))^{\mathcal{M}}$, and since $(\mathcal{P}^{\gamma}(\tilde{A}))^{\mathcal{M}}$ is $\in$-isomorphic to $(\mathcal{P}^{\gamma}(A))^{\mathcal{N}26}$, it follows that there exists $\mathcal{V}\in(\mathcal{P}^{\gamma}(A))^{\mathcal{N}26}$ such that $\mathcal{V}$ is a partition of $A_n$ in $\mathcal{N}26$ into two infinite sets. (Note that $\mathcal{U}=\{\tilde{Z},A_{n}\setminus\tilde{Z}\}$ for some infinite, co-infinite $Z\subseteq A_{n}$, so that $\mathcal{V}=\{Z,A_{n}\setminus Z\}$.) This contradicts the fact that $A_{n}$ is amorphous in $\mathcal{N}26$.

In a similar manner, one can check that the sets $\tilde{A_{n}}$ are pairwise disjoint. To show that $\tilde{Y}$ is denumerable in $\mathcal{M}$, we consider the bijection $f:\omega\to\tilde{Y}$ defined by  $f(n)=\tilde{A}_n$ for $n\in\omega$. We notice that since $|Y|=\aleph_{0}$ in $\mathcal{N}26$, the bijection $n\mapsto A_n$ is in $(\mathcal{P}^{\omega}(A))^{\mathcal{N}26}\subseteq(\mathcal{P}^{\gamma}(A))^{\mathcal{N}26}$. This implies that $f\in (\mathcal{P}^{\omega}(\tilde{A}))^{\mathcal{M}}$. Hence $f\in\mathcal{M}$ and $\tilde{Y}$ is denumerable in $\mathcal{M}$.

Now we assert that $\tilde{Y}$ has no partial multiple choice function in $\mathcal{M}$. Assume the contrary and let $\tilde{Z}$ be an infinite subset of $\tilde{Y}$ with a multiple choice function $f$ in $\mathcal{M}$, i.e. $f$ is a function with domain $\tilde{Z}$ such that for every $z\in\tilde{Z}$, $f(z)$ is a non-empty finite subset of $z$. Since all elements of $\tilde{Z}$ and $[\tilde{A}]^{<\omega}$ are in $(\mathcal{P}(\tilde{A}))^{\mathcal{M}}$, it follows that every element of $f$ is in $(\mathcal{P}^{3}(\tilde{A}))^{\mathcal{M}}$. Thus  $f\in(\mathcal{P}^{4}(\tilde{A}))^{\mathcal{M}}$, and therefore $f\in(\mathcal{P}^{\gamma}(\tilde{A}))^{\mathcal{M}}$. Since $(\mathcal{P}^{\gamma}(\tilde{A}))^{\mathcal{M}}$ is $\in$-isomorphic to $(\mathcal{P}^{\gamma}(A))^{\mathcal{N}26}$, it follows that there exists a $g\in(\mathcal{P}^{\gamma}(A))^{\mathcal{N}26}$ such that $g$ is a multiple choice function for the infinite subset $Z$ of $Y$. This contradicts the fact that $Y$ has no partial multiple choice function in $\mathcal{N}26$.     

The above arguments complete the proof that $\mathbf{\Psi}$ is satisfied in $\mathcal{M}$.
\end{proof}

The following corollary can be deduced from Theorems \ref{s5t:main3}($c$) and \ref{thm:1}:

\begin{corollary}
\label{s03c3}
The statement ``Every dyadically filterbase infinite set is quasi Dedekind-infinite or equipotent to a subset of $\mathbb{R}$'' is independent of $\mathbf{ZF}$.
\end{corollary}

\begin{remark}
\label{s03r4}
By applying Cohen's original model $\mathcal{M}1$ of \cite{hr}, it is easy to show that the statement ``Every dyadically filterbase infinite set is quasi Dedekind-infinite'' is independent of $\mathbf{ZF}$. Namely, it is known that, in $\mathcal{M}1$, the set $A$ of all added Cohen reals is an infinite Dedekind-finite subset of $\mathbb{R}$, so $A$ does not contain any infinite cuf set. By Corollary \ref{s2c06}(iii), it is true in $\mathcal{M}1$ that $A$ is dyadically filterbase infinite.  
\end{remark}

A natural question is whether the implications from Theorem \ref{s5t:main3}($a$) are reversible in $\mathbf{ZF}$. Only the the implication $\mathbf{IFBI}\rightarrow\mathbf{NAS}$ is known to be non-reversible in $\mathbf{ZF}$. Namely, it was shown in \cite{kfbi} that, in the permutation model $\mathcal{N}51$ of \cite{hr}, in which $\mathbf{NAS}$ (but not $\mathbf{BPI}$) is true, $\mathbf{IFBI}$ is false and, in consequence, the conjunction $\mathbf{NAS}\wedge\neg\mathbf{IFBI}$ has a $\mathbf{ZF}$-model. Let us remark that, since $\mathbf{NAS}$ does not imply $\mathbf{BPI}$ in $\mathbf{ZF}$, and the implications $\mathbf{BPI}\rightarrow\mathbf{AC}_{fin}\rightarrow\mathbf{NAS}$ are known to be true in $\mathbf{ZF}$ (see \cite[p. 335]{hr} and \cite[Theorem 4.37 and Proposition 4.39]{her}), $\mathbf{BPI}$ is essentially stronger than $\mathbf{NAS}$ in $\mathbf{ZF}$.  In Theorem \ref{thm:2} below, we provide a substantial strengthening of the result of [15] that $\mathbf{NAS}\wedge\neg\mathbf{IFBI}$ has a $\mathbf{ZF}$-model by establishing that $\mathbf{BPI}\wedge\neg\mathbf{IFBI}$ has a $\mathbf{ZF}$-model. 

\begin{theorem}
\label{thm:2}
\begin{enumerate}
\item[(i)] The Mostowski Linearly Ordered Model $\mathcal{N}3$ in \cite{hr} is a permutation model of $\mathbf{ZFA}$ in which $\mathbf{BPI}\wedge\neg\mathbf{IFBI}$ is true.
\item[(ii)] There exists a $\mathbf{ZF}$-model for $\mathbf{BPI}\wedge\neg\mathbf{IFBI}$. 
\end{enumerate}
\end{theorem}
\begin{proof}
(i) Let us recall in brief a definition of $\mathcal{N}3$. We start with a ground model $M$ of $\mathbf{ZFA}+\mathbf{AC}$ such that, in $M$, there is a linear order $\leq$ on the set $A$ of all atoms of $M$ such that the set $\langle A, \leq\rangle$ is order isomorphic to the set of all rational numbers equipped with the standard linear order. Let $\mathcal{G}$ be the group of all order-automorphisms of $\langle A, \leq\rangle$. Let $\mathcal{I}=[A]^{<\omega}$. Then $\mathcal{N}3$ is the permutation model determined by $M$, $\mathcal{G}$ and the ideal $\mathcal{I}$ (see, e.g., \cite[p. 182]{hr}, \cite[Section 4]{Je} and \cite{ktw0}). It is known that $\mathbf{BPI}$ is true in $\mathcal{N}3$. To show that $\mathbf{IFBI}$ is false in $\mathcal{N}3$, let us prove that the set $A$ is filterbase finite in $\mathcal{N}3$.  By way of contradiction, we assume that $A$ is filterbase infinite in $\mathcal{N}3$. Then, in $\mathcal{N}3$, we can fix a collection $\mathcal{V}=\{\mathcal{V}_{i}:i\in\omega\}$ of free filterbases in $A$ such that, for every pair $i,j$ of distinct elements of $\omega$, there exist $U\in\mathcal{V}_i$ and $V\in\mathcal{V}_j$ with $U\cap V=\emptyset$.

We recall that, for $x\in M$, $\fix_{\mathcal{G}}(x)=\{\phi\in\mathcal{G}: (\forall t\in x) \phi(t)=t\}$ and $\sym_{\mathcal{G}}(x)=\{\phi\in\mathcal{G}: \phi(x)=x\}$. For $x\in M$, a set $E\in\mathcal{I}$ is called a support of $x$ if $\fix_{\mathcal{G}}(x)\subseteq\sym_{\mathcal{G}}(x)$. For $c,d\in A$  with $c<d$, let $(c, d)=\{a\in A: c<a<d\}$, $(-\infty, c)=\{a\in A: a<c\}$ and $(d, +\infty)=\{a\in A: d<a\}$.

Let $E\in\mathcal{I}$ be a support of $\mathcal{V}_{i}$ for all $i\in\omega$. There exists $n\in\omega$ such that $E=\{e_j: j\in n+1\}$ and $e_j<e_k$ if $j\in k\in n+1$. Since $E$ is finite and the filterbases from $\mathcal{V}$ are all free, by defining $\mathcal{V}_i^{\prime}=\{V\setminus E: V\in \mathcal{V}_i\}$ for $i\in\omega$, we obtain in $\mathcal{N}3$ a collection $\mathcal{V}^{\prime}=\{\mathcal{V}_i^{\prime}: i\in\omega\}$ of free filterbases in $A$ such that $E$ is a support of $\mathcal{V}_i^{\prime}$ for every $i\in\omega$. Therefore, without loss of generality, we may assume that, for every $i\in\omega$ and every $V\in\mathcal{V}_{i}$, $V\cap E=\emptyset$. 

Let $\mathcal{Z}=\{(-\infty,e_{0})\}\cup\{(e_{k},e_{k+1}): k\in n\}\cup\{(e_{n},+\infty)\}$. For $I\in\mathcal{Z}$, $i\in\omega$ and $V\in\mathcal{V}_i$, we say that $V$ is bounded in $I$ if there exist $a_I, b_I\in I$ such that $a_I< b_I$ and $V\cap I\subseteq (a_I, b_I)$. If $V$ is not bounded in $I$, we say that $V$ is unbounded in $I$. 

Suppose that $i\in\omega$ and $V\in\mathcal{V}_i$ are such that $V$ is bounded in every $I\in\mathcal{Z}$. For every $I\in\mathcal{Z}$, we fix $a_I, b_I\in I$ such that $a_I<b_I$ and $V\cap I\subseteq (a_I, b_I)$. There exists $\eta\in\fix_{\mathcal{G}}(E)$ such that, for every $I\in\mathcal{Z}$, $\eta((a_I, b_I))\cap (a_I, b_I)=\emptyset$. Clearly, $\eta(\mathcal{V}_i)=\mathcal{V}_i$ and $\eta(V)\cap V=\emptyset$. On the other hand, since $V,\eta(V)\in\mathcal{V}_i$ and $\mathcal{V}_i$ is a filterbase, we have $\eta(V)\cap V\neq\emptyset$. The contradiction obtained shows that the following condition is satisfied:
\begin{equation}
\label{eq:observe1}
(\forall i\in\omega)( \forall V\in\mathcal{V}_{i})(\exists I\in\mathcal{Z}) (V\cap I\text{ is unbounded in } I).
\end{equation} 
In view of this observation and the fact that, for every $i\in\omega$, $\mathcal{V}_{i}$ is a filterbase on $X$, we may assume without loss of generality that 
\begin{equation}
\label{eq:observe2}
(\forall i\in\omega)(\forall V\in\mathcal{V}_{i})(\forall I\in\mathcal{Z})(V\cap I\ne\emptyset\rightarrow V\cap I\text{ is unbounded in $I$}).
\end{equation}
To see this, we fix  $i\in\omega$ and, for $V\in\mathcal{V}_{i}$, we put $\mathcal{U}(V)=\{I\in\mathcal{Z}:V\cap I\text{ is bounded } \}$. We show that, for every $V\in\mathcal{V}_i$, there exists $W\in\mathcal{V}_i$ such that $W\subseteq V$ and, for every $I\in\mathcal{Z}$, if $W\cap I\neq\emptyset$, then $W$ is unbounded in $I$. To do this, we fix $V_0\in\mathcal{V}_i$. Suppose that $\mathcal{U}(V_0)\neq\emptyset$. By (\ref{eq:observe1}), $\mathcal{Z}\setminus\mathcal{U}(V_0)\neq\emptyset$. One can construct a $\psi\in\fix_{G}(E)$ such that, for every $I\in\mathcal{U}(V_0)$, $\psi(V_0\cap I)\cap (V_0\cap I)=\emptyset$ and, for every $I\in\mathcal{Z}\setminus\mathcal{U}(V_0)$, $\psi$ fixes $I$ pointwise (and hence $\psi$ fixes $V_0\cap I$ pointwise). Since $\psi\in\fix_{G}(E)$ and $E$ is a support of $\mathcal{V}_{i}$, we have that $\psi(V_0)\in\psi(\mathcal{V}_{i})=\mathcal{V}_{i}$. Furthermore, $\psi(V_0)\cap V_0=\bigcup\{V_0\cap I:I\in\mathcal{Z}\setminus\mathcal{U}(V_0)\}$ so, for every $I\in\mathcal{Z}$ such that $(\psi(V_0)\cap V_0)\cap I\neq\emptyset$, it is the case that $(\psi(V_0)\cap V_0)\cap I$ is unbounded in $I$. Since $\mathcal{V}_{i}$ is a filterbase, there exists $V_1\in\mathcal{V}_{i}$ such that $V_1\subseteq\psi(V_0)\cap V_0$. We notice that $\{I\in\mathcal{Z}: V_1\cap I\neq\emptyset\}\subseteq \mathcal{Z}\setminus \mathcal{U}(V_0)$ and $\mathcal{U}(V_0)\subseteq\mathcal{U}(V_1)$. If $\mathcal{U}(V_1)\cap\{I\in \mathcal{Z}: V_1\cap I\neq\emptyset\}=\emptyset$, we put $W=V_1$.  If $\mathcal{U}(V_1)\cap\{I\in \mathcal{Z}: V_1\cap I\neq\emptyset\}\neq\emptyset$, in much the same way, as above, we find a set $V_2\in\mathcal{V}_i$ such that $V_2\subseteq V_1$, $\{I\in\mathcal{Z}: V_2\cap I\neq\emptyset\}\subseteq \mathcal{Z}\setminus \mathcal{U}(V_1)$ and $\mathcal{U}(V_1)\subseteq\mathcal{U}(V_2)$. Since $\mathcal{Z}$ is finite, after finitely many steps, for a natural number $k$, we find a set $V_k\in\mathcal{V}_i$ such that $V_k\subseteq V_0$ and, for every $I\in\mathcal{Z}$, if $V_k\cap I\neq\emptyset$, then $V_k$ is unbounded in $I$. This implies that, for every $i\in\omega$,  $\mathcal{V}_i^{\ast}=\{V\in\mathcal{V}_i: (\forall I\in\mathcal{Z})(V\cap I\neq\emptyset\rightarrow V\text{ is unbounded in } I)\}$ is a free filterbase such that $E$ is a support of $\mathcal{V}_i^{\ast}$ and, moreover, if $j\in\omega\setminus\{i\}$, then there exist $U\in\mathcal{V}_i^{\ast}$ and $W\in\mathcal{V}_j^{\ast}$ with $U\cap W=\emptyset$. If $\mathcal{V}$ does not satisfy (\ref{eq:observe2}), we may replace $\{\mathcal{V}_i: i\in\omega\}$ with $\{\mathcal{V}_i^{\ast}: i\in\omega\}$. This is why, for simplicity and without loss of generality, we may assume that (\ref{eq:observe2}) is satisfied. 

Let us notice that it follows from (\ref{eq:observe2}) that we may also assume that the following condition holds:

\begin{enumerate}
\item[(3)] for every $i\in\omega$, every $V\in\mathcal{V}_i$ and every $I\in\mathcal{Z}$, if $V\cap I\neq\emptyset$, then the following conditions are satisfied:   
\begin{enumerate}
\item[(3.1)] if $I=(-\infty,e_{0})$, then there exists $a\in I$ such that $(-\infty,a)\subseteq V$ (in this case, we say that $V$ is of type (3.1.1)) or $(a,e_{0})\subseteq V$ (that is, $V$ is of type (3.1.2));

\item[(3.2)] if $I=(e_{k},e_{k+1})$ for some $k\in n$, then there exists $a\in I$ such that  $(e_{k},a)\subseteq V$(that is, $V$ is of type (3.2.1(k))) or $(a,e_{k+1})\subseteq V$ (that is, $V$ is of type (3.2.2(k)));

\item[(3.3)] if $I=(e_{n},+\infty)$, then there exists $a\in I$ such that $(e_{n},a)\subseteq V$ (this means that $V$ is of type (3.3.1)) or $(a,+\infty)\subseteq V$ (in this case, $V$ is of type (3.3.2)).
\end{enumerate}
\end{enumerate}

\noindent To see this, for $i\in\omega$ and $V\in\mathcal{V}_i$, let $\mathcal{W}(V)$ be the set of all $I\in\mathcal{Z}$ such that $V\cap I\neq\emptyset$ but none of (3.1)-(3.3) is true. Suppose that $i\in\omega$ and $V\in\mathcal{V}_i$ are such that $\mathcal{W}(V)\neq\emptyset$. Then we can find a permutation $\eta\in\fix_{\mathcal{G}}(E)$ such that, for every $I\in\mathcal{W}(V)$,  $\eta(V\cap I)\cap (V\cap I)=\emptyset$ and, for every $I\in\mathcal{Z}\setminus\mathcal{W}(V)$, if $x\in I$, then $\eta(x)=x$.  If $\mathcal{Z}\setminus\mathcal{W}(V)=\emptyset$, then $\eta(V)\cap V=\emptyset$ but this is impossible because $\mathcal{V}_i$ is a filterbase and since $\eta(\mathcal{V}_i)=\mathcal{V}_i$, we have $V,\eta(V)\in\mathcal{V}_i$, so $\eta(V)\cap V\neq\emptyset$ This contradiction shows that $\mathcal{Z}\setminus\mathcal{W}(V)\neq\emptyset$. Since $\eta(V), V\in\mathcal{V}_i$ and $\mathcal{V}_i$ is a filterbase, there exists $W\in\mathcal{V}_i$ such that $W\subseteq\eta(V)\cap V$. Now, using ideas similar to the ones in the paragraph following (\ref{eq:observe2}) above, one can show that, for every $i\in\omega$ and $V\in\mathcal{V}_i$, there exists $U\in\mathcal{V}_i$ such that $U\subseteq V$ and $\mathcal{W}(U)=\emptyset$. We thus take the liberty to leave the details to the interested reader. If $\mathcal{V}$ does not satisfy (3), then we may replace it by a suitable family $\mathcal{V}'=\{\mathcal{V}_{i}':i\in\omega\}$  where, for every $i\in\omega$,  $\mathcal{V}_i^{\prime}=\{V\in\mathcal{V}_i: \mathcal{W}(V)=\emptyset\}$. It is clear that $E$ is a support of every $\mathcal{V}_i^{\prime}$ and $\mathcal{V}^{\prime}$ witnesses that $A$ is filterbase infinite. This is why we may assume that $\mathcal{V}$ satisfies (3). 

Let us fix $m\in\omega$ with $m>2(n+1)$. Since $\mathcal{V}$ witnesses that $A$ is filterbase infinite,  we can chooce a family $\{V_i: i\in m+1\}$ of pairwise disjoint sets such that, for every $i\in m+1$, $V_i\in\mathcal{V}_i$. This is impossible because if $i,j\in N+1$ and $i\neq j$, then $V_i\cap V_j=\emptyset$, so, among the $2(n+1)$ types defined in (3.1)-(3.3), the types of $V_i$ and $V_j$ are different. This completes the proof that $A$ is filterbase finite in $\mathcal{N}3$.

(ii) We have already shown that the statement $\mathbf{BPI}\wedge\neg\mathbf{IFBI}$ has a permutation model.  To see that $\mathbf{BPI}\wedge\neg\mathbf{IFBI}$ can be transferred to a model of $\mathbf{ZF}$ by Theorem \ref{pin}, it suffices to notice that $\neg\mathbf{IFBI}$ is a boundable statement, and thus $\neg\mathbf{IFBI}$ is injectively boundable. 
\end{proof}

\begin{remark}
\label{s2r09}
 We do not know if $\mathbf{IDFBI}$ implies $\mathbf{IWDI}$ in $\mathbf{ZF}$. We also do not know if $\mathbf{ISFBI}$ implies $\mathbf{IDFBI}$ in $\mathbf{ZF}$. In the forthcoming theorem, we show that, in a model of $\mathbf{ZF}$, a weakly Dedekind-finite set can be strongly filterbase infinite.
\end{remark}

\begin{theorem}
\label{thm:3}
The statement ``Every strongly filterbase infinite set is weakly Dedekind-infinite'' is not provable in $\mathbf{ZF}$. 
\end{theorem}
\begin{proof}
Let us denote by $\mathbf{\Phi}$ the statement: There exists a strongly filterbase infinite set which is not weakly Dedekind-infinite. This statement is boundable, so also injectively boundable. Therefore, by Theorem \ref{pin}, to prove that $\mathbf{\Phi}$ has a $\mathbf{ZF}$-model, it suffices to construct a permutation model for $\mathbf{\Phi}$. We remark that, since $\mathbf{\Phi}$ is boundable, if we show that $\mathbf{\Phi}$ has a permutation model, we can also use Jech-Sochor First Embedding Theorem (see \cite[Theorem 6.1 and Problem 1 (p. 94)]{Je}) to deduce that $\mathbf{\Phi}$ has a (symmetric) $\mathbf{ZF}$-model. Let us introduce a new Fraenkel-Mostowski model $\mathcal{N}$ in which $\mathbf{\Phi}$ is true. To this end, we start with a ground model $M$ of $\mathbf{ZFA}+\mathbf{AC}$ with a denumerable set $A$ of atoms, which has a denumerable partition $\{A_{i}:i\in\omega\}$ into infinite sets. For a set $S$ and a mapping $\psi:S\to S$, let $\supp(\psi)=\{ x\in S: \psi(x)\neq x\}$ (i.e., $\supp(\psi)$ is the support of the mapping $\psi$). Let $\mathcal{G}$ be the group of all permutations $\phi$ of $A$ which have the following two properties:
\begin{enumerate}
\item[(1)] for every $i\in\omega$, $\phi$ moves only finitely many elements of $A_{i}$;
\item[(2)] for every $i\in\omega$, there exist $j\in\omega$ and $F\in [A_{j}]^{<\omega}$ such that 
$${\phi[\supp(\phi\upharpoonright A_{i})]=F}.$$
\end{enumerate}
For every $i\in\omega$, we let  
$$\mathcal{Q}_{i}=\{\phi(A_{i}):\phi\in \mathcal{G}\}.$$
 We also let
 $$\mathcal{Q}=\bigcup\{\mathcal{Q}_{i}:i\in\omega\}.$$
 
 For every $E\in [\mathcal{Q}]^{<\omega}$, let $\mathcal{G}_{E}=\{\phi\in \mathcal{G}:{\forall Q\in E}(\phi(Q)=Q)\}$. Then $\mathcal{G}_E$ is a subgroup of $\mathcal{G}$. Since, for $E,E^{\prime}\in [\mathcal{Q}]^{<\omega}$, $\mathcal{G}_{E\cup E^{\prime}}\subseteq\mathcal{G}_E\cap\mathcal{G}_{E^{\prime}}$, the collection $\{\mathcal{G}_E: E\in [\mathcal{Q}]^{<\omega}\}$ is a filterbase in the set of all subgroups of $\mathcal{G}$. We define $\mathcal{F}$ to be the filter of subgroups of $\mathcal{G}$ generated by the filterbase $\{\mathcal{G}_E: E\in [\mathcal{Q}]^{<\omega}\}$.  To see that $\mathcal{F}$ is a normal filter on $\mathcal{G}$ (see \cite[p. 46]{Je} for a definition of a normal filter), we check that $\mathcal{F}$ has the following two properties:
\begin{equation}
\label{eq:normalf1}
\forall a\in A(\{\pi\in \mathcal{G}:\pi(a)=a\}\in\mathcal{F})
\end{equation}
and
\begin{equation}
\label{eq:normalf2}
(\forall\pi\in \mathcal{G})(\forall H\in\mathcal{F})(\pi H\pi^{-1}\in\mathcal{F}).
\end{equation}

To see that (\ref{eq:normalf1}) holds, we fix $a\in A$. There exists a unique $i\in\omega$ such that $a\in A_{i}$. We fix $j\in\omega\setminus\{i\}$ and $a'\in A_{j}$. Let $\phi\in\mathcal{G}$ be the transposition  $(a,a')$. Then $\phi(A_j)=(A_j\setminus\{a^{\prime}\})\cup\{a\}$. The set $E_0=\{A_{i},\phi(A_{j})\}$ is a finite subset of $\mathcal{Q}_{i}\cup\mathcal{Q}_{j}\subset\mathcal{Q}$, so $E_0\in [\mathcal{Q}]^{<\omega}$. Furthermore, it is easy to see that $\mathcal{G}_{E_0}\subseteq \{\pi\in \mathcal{G}:\pi(a)=a\}$. Thus $\{\pi\in \mathcal{G}: \pi(a)=a\}\in\mathcal{F}$, and so (\ref{eq:normalf1}) holds.

To check that (\ref{eq:normalf2}) holds, we fix $\pi\in \mathcal{G}$ and $H\in\mathcal{F}$. There exists $E\in [\mathcal{Q}]^{<\omega}$ such that $\mathcal{G}_{E}\subseteq H$. By the definition of $\mathcal{Q}$, we have $\pi[E]\in[\mathcal{Q}]^{<\omega}$. We assert that $\mathcal{G}_{\pi[E]}\subseteq\pi H\pi^{-1}$. Let $\rho\in \mathcal{G}_{\pi[E]}$. For every $T\in E$ we have the following:
$$\rho(\pi T)=\pi T\rightarrow\pi^{-1}\rho\pi (T)=T;$$
\noindent Hence, since $\mathcal{G}_{E}\subseteq H$, we have:
$$\pi^{-1}\rho\pi\in \mathcal{G}_{E}\rightarrow\rho\in\pi \mathcal{G}_{E}\pi^{-1}\subseteq \pi H\pi^{-1}.$$ 
\noindent Therefore, $\rho\in\pi H\pi^{-1}$. Since $\rho$ is an arbitrary element of $\mathcal{G}_{\pi[E]}$, we conclude that $\mathcal{G}_{\pi[E]}\subseteq\pi H\pi^{-1}$. This implies that $\pi H\pi^{-1}\in\mathcal{F}$, so (\ref{eq:normalf2}) holds.

Let $\mathcal{N}$ be the permutation model determined by $M$, $\mathcal{G}$ and $\mathcal{F}$. We say that an element $x\in\mathcal{N}$ has \emph{support} $E\in [\mathcal{Q}]^{<\omega}$ if, for all $\phi\in \mathcal{G}_{E}$, $\phi(x)=x$. 

To show that the set $A$ of atoms is strongly filterbase infinite in $\mathcal{N}$, for every $i\in\omega$, we define $\mathcal{W}_{i}=\{\bigcap\mathcal{R}:\mathcal{R}\in [\mathcal{Q}_{i}]^{<\omega}\setminus\{\emptyset\}\}$ and  $\mathcal{V}_i=\{\bigcup\mathcal{C}: \mathcal{C}\in[\mathcal{W}_i]^{<\omega}\setminus\{\emptyset\}\}$. We put $\mathcal{V}=\{\mathcal{V}_{i}:i\in\omega\}$. On the basis of the definition of $\mathcal{Q}_{i}$ ($i\in\omega$), it is not hard to verify that any permutation of $A$ in $\mathcal{G}$ fixes $\mathcal{V}$ pointwise. Hence, $\fix_{\mathcal{G}}(\mathcal{V}) =\mathcal{G}\in\mathcal{F}$, and thus $\mathcal{V}$ is well-orderable in the model $\mathcal{N}$  (see also [9, item (4.2), p. 47]). Since $\mathcal{V}$ is denumerable in $M$ and well-orderable in $\mathcal{N}$, it follows that $\mathcal{V}$ is denumerable in $\mathcal{N}$. Let us leave to the interested readers an easy verification that $\mathcal{V}$ satisfies all conditions of Definition \ref{s1d:fbi}($a$). Hence $A$ is strongly filterbase infinite in $\mathcal{N}$. To complete the proof, it remains to show that $A$ is weakly Dedekind-finite in $\mathcal{N}$. 

Suppose that $A$ is weakly Dedekind-infinite in $\mathcal{N}$. Then $(\mathcal{P}(A))^{\mathcal{N}}$ has a denumerable subset $\mathcal{U}=\{U_{n}:n\in\omega\}$ which is in $\mathcal{N}$. Let $E\in [\mathcal{Q}]^{<\omega}$ be a support of $U_{n}$ for all $n\in\omega$.  

Suppose that $\bigcup\mathcal{U}\nsubseteq\bigcup E$. Hence, we can fix $k\in\omega$, $x\in U_{k}\setminus \bigcup E$ and $y\in A\setminus (U_{k}\cup\bigcup E)$.  Let $\psi\in\mathcal{G}$ be the transposition $(x,y)$. Then $\psi\in \mathcal{G}_{E}$, so $\psi(U_{k})=U_{k}$. It follows that $y=\psi(x)\in\psi(U_{k})=U_{k}$, which is impossible. The contradiction obtained shows that $\bigcup\mathcal{U}\subseteq\bigcup E$. This implies that there exist $Z\in E$ and $k,m\in\omega$, such that $k\neq m$ and $ U_{k}\cap Z\neq\emptyset\neq (U_{m}\cap Z)\setminus U_{k}$. We can fix $a\in U_k\cap Z$ and $b\in (U_m\cap Z)\setminus U_k$. 
Let $\phi\in\mathcal{G}$ be the transposition $(a, b)$. Then $\phi\in \mathcal{G}_{E}$, so $\phi(U_{k})=U_{k}$. This implies that $b=\phi(a)\in\phi(U_{k})=U_{k}$, which is impossible. This contradiction shows that $A$ is weakly Dedekind-finite in $\mathcal{N}$.
\end{proof}

\begin{theorem}
\label{thm:4}
Let $\mathcal{N}$ be the permutation model from the proof to Theorem \ref{thm:3}. Then $\mathbf{NAS}$ is false in $\mathcal{N}$. Therefore,  $\mathbf{ISFBI}$ is also false in $\mathcal{N}$.
\end{theorem}
\begin{proof}
Let us use the same notation concerning the definition of $\mathcal{N}$, as in the proof to Theorem \ref{thm:3}. Let $i\in\omega$. We notice that $A_{i}\in\mathcal{N}$ because $\{A_i\}$ is a support of $A_i$.  Suppose that  there exists an infinite set $X\in\mathcal{N}$ such that $X\subseteq A_i$ and $A_i\setminus X$ is infinite. Now, let $E\in [\mathcal{Q}]^{<\omega}$ be a support of $X$. We can fix $x\in X$, $y\in A_{i}\setminus X$ and $\phi\in\mathcal{G}_E$, such that $\phi(x)=y$. Since $E$ is a support of $X$ and $\phi\in \mathcal{G}_{E}$, we have $\phi(X)=X$. However, $y=\phi(x)\in\phi(X)=X$, but this is impossible. The contradiction obtained shows that $A_i$ is amorphous in $\mathcal{N}$, so $\mathbf{NAS}$ is false in $\mathcal{N}$. In view of Theorem \ref{s5t:main3}($a$), $\mathbf{ISFBI}$ implies $\mathbf{NAS}$ in $\mathbf{ZFA}$. Hence $\mathbf{ISFBI}$ is false in $\mathcal{N}$.  
\end{proof}

\begin{proposition}
\label{s2p010}
There exists a model $\mathcal{M}$ of $\mathbf{ZF}$ in which it is true that there exists a locally compact Hausdorff, not completely regular space $\mathbf{X}$ such that all non-empty second-countable compact Hausdorff spaces are remainders of $\mathbf{X}$.
\end{proposition}

\begin{proof} Let $\mathcal{M}$ be any model of $\mathbf{ZF}$ in which there exists a compact Hausdorff not completely regular space $\mathbf{Z}$ (cf, e.g., \cite{gt}) and let us work inside $\mathcal{M}$. It was shown in $\cite{kw0}$ that there exists a Hausdorff compactification $\gamma\mathbf{D}$ of a discrete space $\mathbf{D}$ such that $\gamma D\setminus D$ is homeomorphic to $\mathbf{Z}$. Then assuming that $\mathbb{N}\cap Z=\emptyset$ where $\mathbb{N}$ is the set of all natural numbers of $\mathcal{M}$, we can consider the space $\mathbf{X}=\mathbf{Z}\oplus\mathbb{N}$ which is a Hausdorff locally compact but not completely regular space.  By Proposition \ref{s2p05} and Corollary \ref{s2c06}(ii), all non-empty second-countable compact Hausdorff spaces are remainders of $\mathbf{X}$. 
\end{proof}

\section{The list of open problems}
\label{s6}
To show a direction of future research relevant to the topic of this paper and for the convenience of readers, we include a shortlist of open problems below. 
\begin{enumerate}
\item[(1)] Is there a model of $\mathbf{ZFA}$ (or of $\mathbf{ZF}$) in which $\mathbf{ISFBI}$ is true but $\mathbf{IWDI}$ is false?
\item[(2)] Is there a model of $\mathbf{ZFA}$ (or of $\mathbf{ZF}$) in which $\mathbf{ISFBI}$ is true but $\mathbf{IDFBI}$ is false?
\item[(3)] Is there a model of $\mathbf{ZFA}$ (or of $\mathbf{ZF}$) in which $\mathbf{IDFBI}$ is true but $\mathbf{IWDI}$ is false?
\item[(4)] Is there a model of $\mathbf{ZFA}$ (or of $\mathbf{ZF}$) in which a filterbase infinite set can be strongly filterbase finite?
\item[(5)] Is there a model of $\mathbf{ZFA}$ (or of $\mathbf{ZF}$) in which a strongly filterbase infinite set can be dyadically filterbase finite?
\item[(6)] Is there a model of $\mathbf{ZFA}$ (or of $\mathbf{ZF}$) in which a dyadically filterbase infinite set can be weakly Dedekind-finite?
\item[(7)]  Is it provable in $\mathbf{ZF}$ that if an infinite discrete space $\mathbf{D}$  has a denumerable remainder, then $\mathbb{N}(\infty)$ is a remainder of $\mathbf{D}$?
\item[(8)] Is it provable in $\mathbf{ZF}$ that if a locally compact Hausdorff space $\mathbf{X}$ has a denumerable remainder, then $\mathbb{N}(\infty)$ is a remainder of $\mathbf{X}$?
\item[(9)] Is is provable in $\mathbf{ZF}$ that if a compact Hausdorff space $\mathbf{X}$ has a cuf base, then $\mathbf{X}$ is weakly Loeb?
\item[(10)] Is it provable in $\mathbf{ZF}$ that if a compact Hausdorff space has a cuf base and a Hausdorff compactification $\gamma\mathbf{X}$ has a second-countable remainder, then $\gamma\mathbf{X}$ is metrizable?
\end{enumerate}

Clearly, a positive answer to question (3) gives a positive answer to question (1). A negative answer to question (7) gives a negative answer to question (8).


\begin{thebibliography}{BOP} 
\normalsize 
\baselineskip=17pt

\bibitem{br} N. Brunner, \emph{Products of compact spaces in the least permutation model},
Z. Math. Logik Grundlagen Math. 31 (1985) 441--448.

%\bibitem{CP} E. \v Cech, B. Pospi\v sil, \emph{Sur les espaces compact}, Publ. Fac. Sci. Univ. Masaryk 258 (1938) 1--14.

\bibitem{ch} R. E. Chandler, \emph{Hausdorff compactifications}, Marcel Dekker, New York, 1978.

\bibitem{En} R. Engelking, \textit{General Topology}, Sigma Series in Pure Mathematics 6, Heldermann, Berlin, 1989. 

\bibitem{gt} C. Good and I. Tree, \emph{Continuing horrors of topology
without choice}, Topology Appl. 63 (1995) 79--90.

\bibitem{hm0} J. Hatzenbuhler, D. A.  Mattson, \emph{Spaces for which all compact metric spaces are remainders}, Proc. Amer. Math. Soc. 82(3) (1981) 478--480.

\bibitem{her} H. Herrlich, \emph{Axiom of Choice}, Lecture Notes Math., vol. 1876, Springer, New York, 2006.

\bibitem{hr} P. Howard and J. E. Rubin, \textit{Consequences of the axiom of
choice, }Math. Surveys and Monographs 59, A.M.S., Providence R.I., 1998.

\bibitem{howtach} P. Howard and E. Tachtsis,  \emph{On metrizability and compactness of certain
products without the axiom of choice}, submitted.

\bibitem{Je}T. J. Jech, \textit{The Axiom of Choice}, Studies in Logic and the Foundations of Mathematics, \textbf{75}, North-Holland, Amsterdam, 1973.

\bibitem{k02}  K. Keremedis, \textit{Consequences of the failure of the axiom of choice in
the theory of Lindel\"of metric spaces}, Math. Logic Quart. 50 (2004),
no. 2, 141--151.

\bibitem{k01} K. Keremedis, \emph{On the relative strength of forms of compactness of metric spaces and their countable productivity in ZF}, Topology Appl. 159 (2012), 3396--3403.

\bibitem{kk} K. Keremedis, \textit{Compact and Loeb Hausdorff spaces in }$%
\mathbf{ZF}$\textit{\ and the axiom of choice for families of finite sets},
Math. Logic Quart. \textbf{58}, no. 3, (2012) 130 -- 138.

\bibitem{k} K. Keremedis, \textit{On sequentially compact and related
notions of compactness of metric spaces in} $\mathbf{ZF}$, Bull. Pol.
Acad. Sci. Math. 64 (2016) 29--46.

\bibitem{knd} K. Keremedis, \textit{Non-discrete metrics in $\mathbf{ZF}$ and some notions of finiteness}, Math. Logic Quart., 62 (2016), 383--390.

\bibitem{kfbi} K. Keremedis, \textit{Extending $T_1$ topologies to Hausdorff with the same sets of limit points}, Topology Proc. 49 (2017) 121--133.

\bibitem{kerta} K. Keremedis and E. Tachtsis, \emph{On Loeb and weakly Loeb Hausdorff spaces}, Sci. Math. Jpn. Online 4 (2001) 15--19.

\bibitem{kt} K. Keremedis and E. Tachtsis, \emph{Countable sums and products of
metrizable spaces in $\mathbf{ZF}$}, Math. Logic Quart. 51 (2005) 95--103.

\bibitem{keremTac} K. Keremedis and E. Tachtsis, \textit{Countable compact
Hausdorff spaces need not be metrizable in }$\mathbf{ZF}$, Proc.
Amer. Math. Soc. 135 (2007), 1205-1211.

\bibitem{kttych} K. Keremedis and E. Tachtsis,\textit{Products of some special compact
spaces and restricted forms of AC}, J. Symbolic Logic 75(2) (2010), 996--1006.

\bibitem{ktcell} K. Keremedis and E. Tachtsis, \textit{Cellularity of infinite Hausdorff spaces in $\mathbf{ZF}$}, Topology Appl. 274 (2020) 107104

\bibitem{ktw0} K. Keremedis, E. Tachtsis and E. Wajch, \emph{Several amazing discoveries about compact metrizable spaces in $\mathbf{ZF}$}, new manuscript.

\bibitem{kw0} K. Keremedis and E. Wajch, \emph{Hausdorff compactifications in $\mathbf{ZF}$}, Topology Appl. 258 (2019) 79--99.

\bibitem{kw2} K. Keremedis and E. Wajch, \emph{On Loeb and sequential spaces in $\mathbf{ZF}$}, Topology Appl. 280 (2020) 107279.

\bibitem{kw1} K. Keremedis and E. Wajch, \emph{Cuf products and cuf sums of (quasi)-metrizable spaces in $\mathbf{ZF}$}, submitted, preprint available at http://arxiv.org/abs/2004.13097

\bibitem{lo} P. A. Loeb, \textit{A new proof of the Tychonoff theorem}, Amer. Math. Monthly
72 (1965) 711--717.

\bibitem{mag} K. D. Magill, Jr., \textit{A note on compactifications}, Math. Z. 94 (1966), 322--325.

\bibitem{pw} A. Pi\k{e}kosz and E. Wajch, \textit{Compactness and compactifications in generalized topology}, Topology Appl. 94 (2015) 241--268.

\bibitem{pin1}  D. Pincus, \textit{Zermelo-Fraenkel consistency results by Fraenkel-Mostowski
methods},  J. Symbolic Logic 37 (1972), 721--743.

\bibitem{pin2}  D. Pincus, \textit{Adding Dependent Choice}, Annals Math. Logic 11 (1977), 105--145.

\bibitem{ew} E. Wajch, \textit{Quasi-metrizability of products in} $\mathbf{%
ZF}$ \textit{and equivalences of} $\mathbf{CUT}$(fin), Topology Appl. 241 (2018) 62--69.

\bibitem{w} S. Willard, \emph{General Topology}, Addison-Wesley Series in Math., Addison-Wesley Publishing Co., Reading, Massachusetts ,1968.



\end{thebibliography}
\end{document}